\documentclass[11pt]{amsart}

\usepackage{amsfonts, amsmath, amssymb, color}
\usepackage{amsthm}
\usepackage{bbm}
\usepackage{booktabs}
\usepackage{float}
\usepackage{hhline}
\usepackage{lipsum}
\usepackage{lscape}
\usepackage{subfig}
\usepackage{textcomp}
\usepackage{tikz}
\usetikzlibrary{decorations.pathmorphing,shapes,arrows,positioning}
\usepackage{xcolor}
\usepackage{young}
\usepackage{youngtab}
\usepackage{ytableau}
\usepackage{tikz}
\usepackage{float}
\usepackage{scalefnt}
\usepackage{hyperref}
\usetikzlibrary{calc}

\hypersetup{colorlinks,linkcolor=blue,urlcolor=cyan,citecolor=blue}
\allowdisplaybreaks[4]
\theoremstyle{plain}
\newtheorem{thm}{Theorem}[section]
\newtheorem{lem}[thm]{Lemma}
\newtheorem{prop}[thm]{Proposition}
\newtheorem{cor}[thm]{Corollary}
\newtheorem{rem}[thm]{Remark}
\theoremstyle{definition}

\newtheorem{exmp}[thm]{Example}

\makeatletter
\newcommand{\rmnum}[1]{\romannumeral #1}
\newcommand{\Rmnum}[1]{\expandafter\@slowromancap\romannumeral #1@}
\makeatother
\newcommand{\la}{\lambda}

\numberwithin{equation}{section} \errorcontextlines=0
\begin{document}
\title{A spin analog of the plethystic Murnaghan-Nakayama rule}
\author{Yue Cao}
\address{School of Mathematics, South China University of Technology, Guangzhou, Guangdong 510640, China}
\email{434406296@qq.com}
\author{Naihuan Jing}
\address{Department of Mathematics, North Carolina State University, Raleigh, NC 27695, USA}
\email{jing@ncsu.edu}
\author{Ning Liu}
\address{School of Mathematics, South China University of Technology, Guangzhou, Guangdong 510640, China}
\email{mathliu123@outlook.com}
\subjclass[2010]{Primary: 05E10, 05E05; Secondary: 17B69}\keywords{plethysm, Murnaghan-Nakayama rule, Schur $Q$-functions, Hall-Littlewood functions, vertex operators}

\maketitle
\begin{abstract} As a spin analog of the plethystic Murnaghan-Nakayama rule for Schur functions,
   the plethystic Murnaghan-Nakayama rule for Schur $Q$-functions is established with the help of the vertex operator realization. This generalizes both
   the Murnaghan-Nakayama rule and the Pieri rule for Schur $Q$-functions. A plethystic Murnaghan-Nakayama rule for Hall-Littlewood functions is also investigated.
\end{abstract}
\tableofcontents

\section{Introduction}
Schur functions form an important and well studied basis in the algebra $\Lambda$ of symmetric functions, and one application is to
compute the complex irreducible characters of the symmetric group via Frobenius' formula. The subalgebra of $\Lambda$ generated by the odd-degree power sums
 also has a distinguished basis of Schur's $Q$-functions. In the classic work \cite{S}, Schur showed
that the nontrivial irreducible complex projective (or spin) characters of the symmetric group
are given by the Schur Q-functions. Like Schur functions, the Schur Q-functions can also be defined combinatorially in terms of shifted tableaux
\cite{St}, and they can also be realized by (twisted) vertex operators \cite{J1}.

Plethysm of symmetric functions was introduced by Littlewood (see \cite{L1}) as the composition of representations of general linear groups (cf. \cite{M}).
 The plethystic Murnaghan-Nakayama rule for Schur functions was introduced in \cite{DLT} using Muir's rule. Let $p_{n}$, $h_{k}$ and $s_{\mu}$ denote the $n$-th power sum symmetric function, the complete symmetric function and the Schur function respectively. Then the combinatorial rule says that
\begin{align}\label{e:c}
(p_n\circ h_k)s_{\mu}=\sum_{\lambda}(-1)^{spin(\lambda/\mu)}s_{\lambda}
\end{align}
summed over all partitions $\lambda$ such that $\lambda/\mu$ is a horizontal $n$-border strip of weight $k$
and $spin(\lambda/\mu)$ counts the total number of rows minus the number of components (cf. \cite[p.3]{CJL} for the explicit definitions).
When $k=1$, it reduces
to the usual Murnaghan-Nakayama rule, 
and when $n=1$ it specializes to the Pieri rule. 
An alternative proof of \eqref{e:c} was provided in \cite{EPW} using the character theory of the symmetric group
and a direct combinatorial proof was given by Wildon \cite{W}. Using the vertex operator realization of the Schur functions, we
have recently obtained a determinant-type plethystic Murnaghan-Nakayama rule \cite{CJL}, which includes \eqref{e:c} as a consequence.

The purpose of this paper is to establish a spin analog of the plethystic Murnaghan-Nakayama rule for Schur $Q$-functions.
This is achieved by formulating the coefficients of the adjoint action of
 $p_s\circ Q_{(k)}$ on the Schur Q-function $Q_{\lambda}$ associated to strict partition $\lambda$ via Pfaffians.
With help of the Pfaffian-type plethystic Murnaghan-Nakayama rule for Schur Q-functions, we obtain the plethystic Murnaghan-Nakayama rule for Schur $Q$-functions:
\begin{align}\label{e:combin.}
(p_s\circ Q_{(k)})Q_{\mu}=\sum\limits_{\lambda}sgn(\sigma)2^{A(\lambda/\mu)}2^{l(\mu)-l(\lambda)}Q_{\lambda}
\end{align}
where the sum runs over all strict partitions $\lambda\supset\mu$ such that $\lambda/\mu$ is a symmetric horizontal $(s,k)$-strip (see Corollary \ref{t:combin.}). Here the length $l(\la)$ of partition $\la$ is defined in section 2, the symmetric horizontal $(s,k)$-strip, the nonnegative integer $A(\lambda/\mu)$ and the permutation $\sigma$ associated to $(\lambda, \mu)$ are explained in details in section 3.

The layout of the paper goes as follows. In Section 2, we review some basic concepts and preliminaries about symmetric functions. In Section 3, we begin by recalling the vertex operator realization of the Schur $Q$-functions. We then introduce the vertex operator $T^{(s)}_k$ to realize the plethystic action of $p_s\circ Q_{(k)}$ on symmetric functions.
Using the Wick formula, we compute the adjoint action $T^{(s)-}_{k}Q_{\lambda}$ via Pfaffians (see Theorem \ref{t:Q}). Based on these computations, the plethystic Murnaghan-Nakayama rule is established. Section 4 is devoted to a plethystic Murnaghan-Nakayama rule for Hall-Littlewood functions
as generalization of previous sections.

\section{\textbf{Preliminaries about symmetric functions}}
In this section, we recall some notations and definitions following \cite[Ch. I, \S 1-2; Ch. III]{Mac}. A composition $\lambda=\left(\lambda_{1}, \ldots, \lambda_{l}\right)$ is a list of nonnegative integers, and denoted as $\lambda\models|\lambda|=\sum_{i=1}^{l}\lambda_{i}$.  A composition $\lambda$ becomes a partition of $|\lambda|$ if the parts $\lambda_i$ are ordered: $\lambda_1\geq \lambda_2\geq \ldots \geq \lambda_l$, denoted as $\lambda \vdash|\lambda|$. The number $l(\lambda)$ of positive parts of $\lambda$ is called the length of $\lambda$.
Sometime we write $\lambda=\left(\cdots, 2^{m_{2}(\lambda)}, 1^{m_{1}(\lambda)}\right)$, where $m_{i}(\lambda)$ is the multiplicity of
part $i$ in $\lambda$.
For $\lambda\vdash n$ denote $z_{\lambda}=\prod_{i\geq  1}i^{m_{i}(\lambda)}m_{i}(\lambda)!$ and
\begin{align*}
z_{\lambda}(t)=\frac{z_{\lambda}}{\prod_{i\geq 1}(1-t^{i})^{m_{i}(\lambda)}}.
\end{align*}
An odd (resp. strict) partition is a partition with odd (resp. distinct) parts. Let $\mathcal{OP}$ (resp. $\mathcal{SP}$) be the set of all odd (resp. strict) partitions, and let $\mathcal{OP}_{n} =\{\lambda\in \mathcal{OP}: |\lambda|=n\}$ and $\mathcal{SP}_n=\{\lambda\in\mathcal{SP}: |\lambda|=n\}$. Sometimes we write $\lambda \vdash_{o} n$ (resp. $\lambda\vdash_s n$) to mean $\lambda\in \mathcal{OP}_n$ (resp. $\lambda\in\mathcal{SP}_n$).

A partition is identified with its {\it Young diagram}, which is the left-justified array of squares consisting of $\la_i$ squares in row $i$ with $i$ increasing downward. For instance,
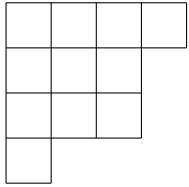
\begin{gather*}
  \centering
\begin{tikzpicture}[scale=0.6]
   \coordinate (Origin)   at (0,0);
    \coordinate (XAxisMin) at (0,0);
    \coordinate (XAxisMax) at (4,0);
    \coordinate (YAxisMin) at (0,-3);
    \coordinate (YAxisMax) at (0,0);
\draw [thin, black] (0,0) -- (4,0);
    \draw [thin, black] (0,-1) -- (4,-1);
    \draw [thin, black] (0,-2) -- (3,-2);
    \draw [thin, black] (0,-3) -- (3,-3);
    \draw [thin, black] (0,-4) -- (1,-4);
    \draw [thin, black] (0,0) -- (0,-4);
    \draw [thin, black] (1,0) -- (1,-4);
    \draw [thin, black] (2,0) -- (2,-3);
    \draw [thin, black] (3,0) -- (3,-3);
    \draw [thin, black] (4,0) -- (4,-1);
    \end{tikzpicture}
\end{gather*}
is the Young diagram of $(4,3,3,1)$. We say $\mu\subset\la$ to mean $\mu_i\leq\la_i$ for any $i$. If $\mu\subset\la$, the set-theoretic difference $\theta=\la-\mu$ is called a {skew diagram}. For example, if $\la=(4,3,3,1)$ and $\mu=(3,2,1)$, the skew diagram $\theta=\la-\mu$ is the shaded region in the picture below.
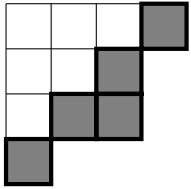
\begin{gather*}
  \centering
\begin{tikzpicture}[scale=0.6]
   \coordinate (Origin)   at (0,0);
    \coordinate (XAxisMin) at (0,0);
    \coordinate (XAxisMax) at (4,0);
    \coordinate (YAxisMin) at (0,-3);
    \coordinate (YAxisMax) at (0,0);
\draw [thin, black] (0,0) -- (4,0);
    \draw [thin, black] (0,-1) -- (4,-1);
    \draw [thin, black] (0,-2) -- (3,-2);
    \draw [thin, black] (0,-3) -- (3,-3);
    \draw [thin, black] (0,-4) -- (1,-4);
    \draw [thin, black] (0,0) -- (0,-4);
    \draw [thin, black] (1,0) -- (1,-4);
    \draw [thin, black] (2,0) -- (2,-3);
    \draw [thin, black] (3,0) -- (3,-3);
    \draw [thin, black] (4,0) -- (4,-1);
    \filldraw[fill = gray][ultra thick]
    (3,0) rectangle (4,-1) (2,-1)rectangle(3,-2) (2,-2)rectangle(3,-3) (1,-2)rectangle(2,-3) (0,-3)rectangle(1,-4);
    \end{tikzpicture}
\end{gather*}
We define the conjugate partition $\la^{'}$ by reflecting the diagram of $\la$ along the diagonal, so that the conjugate of $(4,3,3,1)$ above is still $(4,3,3,1)$.

Let $\Lambda_{\mathbb{Q}}$ be the ring of symmetric functions in infinitely many variables
$x_{1},x_{2},\ldots $ over $\mathbb{Q}$, and let $\Lambda_{\mathbb{Q}(t)}=\Lambda_{\mathbb{Q}}\otimes \mathbb{Q}(t)$. Let  $p_{n}=\sum x^n_{i}$ be the degree $n$ power sum symmetric function. Then $p_{\lambda}=p_{\lambda_{1}}p_{\lambda_{2}}\cdots p_{\lambda_{l}}$ for $\lambda=\left(\lambda_{1}, \ldots, \lambda_{l}\right)\in\mathcal P$ form a $\mathbb{Q}(t)$-basis of $\Lambda_{\mathbb{Q}(t)}$. The space $\Lambda_{\mathbb{Q}(t)}$ is equipped with the Hall-Littlewood bilinear form $\langle\cdot,\cdot\rangle_{t}$ given by \cite[p.225]{Mac}.
\begin{align} \label{e:pp}
\langle p_{\lambda},p_{\mu}\rangle_t=\delta_{\lambda\mu}z_{\lambda}(t),
\end{align}
under which the Hall-Littlewood functions $Q_{\lambda}(t)$ form an orthogonal basis \cite[Ch.III, (2.11),(4.9)]{Mac}:
\begin{align}
\langle Q_{\lambda}(t), Q_{\mu}(t)\rangle_{t}=\delta_{\lambda\mu}b_{\lambda}(t),
\end{align}
where $b_{\lambda}(t)=(1-t)^{l(\lambda)}\prod_{i\geqslant 1}[m_i(\lambda)]!$, $[n]=\frac{1-t^n}{1-t}$, and $[n]!=[n][n-1]\cdots[1]$.

Define the $\mathbb{Q}(t)$-linear and anti-involutive automorphism $\ast$ by
$\langle fg,h\rangle_t=\langle g,f^{\ast}h\rangle_t$
for any $f,g,h\in \Lambda_{\mathbb{Q}(t)}$. It is clear from \eqref{e:pp} that the adjoint $p^*_m$ is the differential operator $p^{\ast}_{m}=\frac{m}{1-t^{m}}\frac{\partial}{\partial p_{m}}$ for $m\in\mathbb N$.

Let $\Gamma_{\mathbb{Q}}$ be the subring of $\Lambda_{\mathbb{Q}}$ generated by $p_{1},p_{3},p_5,\ldots $. It is well-known that the space $\Gamma_{\mathbb{Q}}$ is spanned by the Schur $Q$-functions $Q_{\lambda}$ ($\lambda$ strict) \cite[p.253]{Mac}.
Similar to \eqref{e:pp}, $\Gamma_{\mathbb{Q}}$ has the following inner product \cite[p.255]{Mac}
\begin{align}\label{e:Qinner}
\langle p_{\lambda}, p_{\mu}\rangle_{-1}=\delta_{\lambda \mu}\frac{z_{\lambda}}{2^{l(\lambda)}},\quad \lambda, \mu\in \mathcal{OP}.
\end{align}
Under which the $Q_{\lambda}$ ($\lambda$ strict) forms an orthogonal basis such that \cite[p.255]{Mac}
\begin{align}
\langle Q_{\lambda}, Q_{\mu}\rangle_{-1}=2^{l(\lambda)} \delta_{\lambda \mu}, \quad \lambda, \mu\in \mathcal{SP}.
\end{align}

For odd integer $m$, the multiplication operator $p_{m}$ and its adjoint operator with respect to \eqref{e:Qinner}
is given by $p_{m}^{-}=\frac{m}{2} \frac{\partial}{\partial p_{m}}$ for positive odd integer $m$. So $p_m^-$ can be viewed
as the specialization of $p_m^*$ at $t=-1$.

An important further operation on symmetric functions is the plethysm \cite[Ch. X.II]{L1}\cite[Ch.I, \S 8]{Mac}.
 For $f\in\Lambda_{\mathbb Q}$, the plethysm $g \mapsto g \circ f$ on any  $g=\sum_{\lambda} c_{\lambda} p_{\lambda}\in\Lambda_{\mathbb Q}$ means
\begin{align}
g\circ f=\sum_{\lambda} c_{\lambda} \prod_{i=1}^{l(\lambda)} f\left(x_{1}^{\lambda_{i}}, x_{2}^{\lambda_{i}}, \ldots\right).
\end{align}
Thus the plethysm by $f$ is an algebra isomorphism of $\Lambda$ sending $p_{k} \mapsto$ $f\left(x_{1}^{k}, x_{2}^{k}, \ldots\right)$.

Following \cite[Sec.6]{Bak}, we define a $t$-analog of plethysm in $\Lambda_{\mathbb{Q}(t)}$ as follows. For $f \in \Lambda_{\mathbb{Q}(t)}$,
the plethysm $g \mapsto g\diamond f$ on any $g=\sum_{\lambda} c_{\lambda}(t) p_{\lambda}\in \Lambda_{\mathbb{Q}(t)}$ means
\begin{align}
g\diamond f=\sum_{\lambda} c_{\lambda}(t) \prod_{i=1}^{l(\lambda)} f\left(x_{1}^{\lambda_{i}}, x_{2}^{\lambda_{i}}, \ldots; t^{\lambda_i}\right).
\end{align}
Thus the plethysm by $f$ is the algebra isomorphism of $\Lambda_{\mathbb{Q}(t)}$ sending $p_{n} \mapsto$ $f\left(x_{1}^{n}, x_{2}^{n}, \ldots ; t^{n}\right)$. For instance, the plethysm by $(1+t) p_{1}$ is given by sending $p_{n}\mapsto (1+t^{n}) p_{n}$.

Let $f, g, h$ be symmetric functions. We have the following properties for plethysm \cite[Ch.I, \S 8]{Mac}:
\begin{align}
(af+bg)\circ h&=a(f\circ h)+b(g\circ h),\quad a,b\in\mathbb{Q},\\
(fg)\circ h&=(f\circ h)(g\circ h),\\
f\circ p_k&=p_k \circ f.
\end{align}

\section{\textbf{A plethystic Murnaghan-Nakayama rule for Schur's $Q$-functions}}
\subsection{\bf{Vertex operator realization of the Schur $Q$-functions}}

Recall the vertex operator realization of the Schur $Q$-functions \cite[Thm. 5.9]{J1}. Define the vertex operator $Q(z)$ and its adjoint operator $Q^{-}(z)$ (cf. the vertex operators for Hall-Littlewood functions at $t = -1$ in \cite[Eqs. 2.5, 2.6]{JL2}) by
\begin{align}
Q(z)&=\exp \left(\sum_{m\geq1,\text{odd}}\frac{2}{m} p_{m} z^{m}\right) \exp \left(-\sum_{m\geq1,\text{odd}} \frac{\partial}{\partial p_{m}} z^{-m}\right)=\sum_{m \in \mathbb{Z}} Q_{m} z^{m},\\
Q^{-}(z)&=\exp \left(-\sum_{m\geq1,\text{odd}}\frac{2}{m} p_{m} z^{m}\right) \exp \left(\sum_{m\geq1,\text{odd}} \frac{\partial}{\partial p_{m}} z^{-m}\right)=\sum_{m \in \mathbb{Z}} Q_{m}^{-} z^{-m},
\end{align}
which act on the ring of symmetric functions via the linear basis of power-sum symmetric functions.

Introduce symmetric functions $q_{m}$ as follows
\begin{align}\label{t:q(z)}
q(z)=\exp \left(\sum_{m\geq1,\text{odd}} \frac{2}{m} p_{m} z^{m}\right)=\sum_{m \geq 0} q_{m} z^{m}.
\end{align}
We call $ q_{m}$ the Schur $Q$-function associated to the one-row partition $(m)$. It follows from (\ref{t:q(z)}) that
\begin{align}\label{e:basic}
q_{m}=\sum_{\lambda\vdash_{o} m} \frac{2^{l(\lambda)}}{z_{\lambda}} p_{\lambda} \quad(m>0).
\end{align}
Define $q_{\lambda}=q_{\lambda_{1}} q_{\lambda_{2}} \cdots q_{\lambda_{l}}$ for any partition $\lambda=(\lambda_{1},\ldots,\lambda_{l})$. Then $q_{\lambda}$ ($\lambda$ odd) form a basis of $\Gamma_{\mathbb{Q}}.$

For any strict partition $\mu=(\mu_{1},\mu_{2},\ldots,\mu_{l})$, by \cite[Prop. 2.17, Eq.(3.8) for $t=-1$]{J2} and \cite[p.253]{Mac} we have
\begin{align}\label{e:Schur}
Q_{\mu_{1}}Q_{\mu_{2}}\cdots Q_{\mu_{l}}.1=\prod\limits_{i<j}\frac{1-R_{ij}}{1+R_{ij}}q_{\mu_{1}}q_{\mu_{2}}\cdots q_{\mu_{l}}=Q_{\mu}
\end{align}
where $Q_{\mu}$ is the Schur $Q$-function indexed by $\mu$ and $R_{ij}$ is the raising operator defined as $R_{ij}q_{\mu}=q_{(\mu_{1},\ldots ,\mu_{i}+1,\ldots ,\mu_{j}-1,\ldots , \mu_{l})}$.
\begin{prop} \cite[Props. 4.15, 4.16]{J1}\label{t:characterization} The components of $Q(z)$ and $Q^-(z)$ generate a Clifford algebra with the following relations
\begin{align}\label{e:exchange}
\left\{Q_{m}, Q_{n}^{-}\right\}=2 \delta_{m, n}; \quad \left\{Q_{m}, Q_{n}\right\}=\left\{Q_{m}^{-}, Q_{n}^{-}\right\}=(-1)^{n} 2 \delta_{m,-n}
\end{align}
where $\{A, B\}:=AB+BA$. Moreover, $Q_{-m}.1=Q_m^-.1=\delta_{m,0} \,  (m\geq0)$.
\end{prop}
From now on, {\it we fix $s$ as an odd positive integer} and $k$ as a positive integer unless specified otherwise.
Introduce the vertex operator $T^{(s)}(z)$ as follows:
\begin{align}
T^{(s)}(z)&=\exp \left(2\sum_{m\geq1,\text{odd}} \frac{p_{sm}}{m} z^{m}\right)=\sum_{m \geq 0} T^{(s)}_{m} z^{m}.
\end{align}
Note that $p_{m}^{-}=\frac{m}{2}\frac{\partial}{\partial p_m}$. The adjoint operator $T^{(s)-}(z)$ can be written as
\begin{align}
T^{(s)-}(z)&=\exp \left(\sum_{m\geq1,\text{odd}}s\frac{\partial}{\partial p_{sm}}z^{-m}\right)=\sum_{m \geq 0} T^{(s)-}_{m} z^{-m}.
\end{align}

The following facts are clear.
\begin{align}
T^{(s)-}_{m}=0 \,\, (m<0),\quad T^{(s)-}_{m}.1=\delta_{m,0} \, \, (m\geq 0);\\
T^{(s)}_{m}=\sum_{\lambda\vdash_{o} m}\frac{2^{l(\lambda)}p_{s\lambda}}{z_{\lambda}}=p_{s}\circ\sum_{\lambda\vdash_{o} m} \frac{2^{l(\lambda)}}{z_{\lambda}} p_{\lambda}=p_{s}\circ q_{m}
\end{align}
\begin{thm}\label{t:Q}  For a strict partition $\lambda=(\lambda_{1},\ldots,\lambda_{l})$, we have that
\begin{align}\label{e:Q}
T^{(s)-}_{k}Q_{\lambda}.1=\sum_{\nu\models k}2^{l(\nu)}Q_{\lambda-s\nu}.1,
\end{align}
where $\lambda-s\nu=(\lambda_{1}-s\nu_{1},\lambda_{2}-s\nu_{2},\ldots)$.
\end{thm}
\begin{proof} For any two operators $A$ and $B$, if $[A,B]$ commutes with both $A$ and $B$, by \cite[Thm. 5.1]{H}, we have that
\begin{align*}
\exp(A)\exp(B)=\exp(B)\exp(A)\exp[A,B].
\end{align*}
(Indeed, \cite[Thm. 5.1]{H} is applied once to the left-hand-side
and twice to the right-hand-side).

It follows that for $|z|>|w|$
\begin{align*}
T^{(s)-}(z)Q(w)&=Q(w)T^{(s)-}(z)\exp \left(\sum_{m\geq1,\text{odd}}\left[s\frac{\partial}{\partial p_{sm}},\frac{2}{sm}p_{sm}\right](\frac{w^{s}}{z})^{m}\right)\\
&=Q(w)T^{(s)-}(z)\exp \left(\sum_{m\geq1,\text{odd}}\frac{2}{m}(\frac{w^{s}}{z})^{m}\right).
\end{align*}
By Taylor's expansion $\ln\frac{1+x}{1-x}=\ln(1+x)-\ln(1-x)=\sum_{m\geq1,\text{odd}}\frac{2}{m}x^{m}$ for $|x|<1$. Thus,
$$\exp \left(\sum_{m\geq1,\text{odd}}\frac{2}{m}(\frac{w^{s}}{z})^{m}\right)=\frac{z+w^{s}}{z-w^{s}},$$
which yields
$$T^{(s)-}(z)Q(w)=Q(w)T^{(s)-}(z)\frac{z+w^{s}}{z-w^{s}}=Q(w)T^{(s)-}(z)\left(1+2 \sum_{k \geq 1}\left(w^s / z\right)^k\right).$$
Taking coefficients of $z^{-k}w^{m}$, we have
\begin{align*}
T^{(s)-}_{k}Q_{m}=Q_{m}T^{(s)-}_{k}+2\sum_{r=1}^{k}Q_{m-sr}T^{(s)-}_{k-r}.
\end{align*}
Then \eqref{e:Q} follows by induction on $l(\lambda)$.
\end{proof}

\subsection{\bf Computation of $T^{(s)-}_{k}Q_{\lambda}.1$ via Pfaffians.}

Recall that the Pfaffian Pf$(A)$
of an antisymmetric matrix $A=(a_{i,j})_{2n\times2n}$ satisfies $\text{Pf(A)}^{2}=\det(A)$ and it is explicitly given by
\begin{align}
\text{Pf(A)}=\sum\limits_{\sigma} sgn(\sigma)\prod _{i=1}^{n}a_{\sigma(2i-1)\sigma(2i)}
\end{align}
summed over all $2$-shuffles of $\{1,2,\cdots,2n\},$ i.e. $ \sigma\in\mathfrak{S}_{2n}$ such that $\sigma(2i-1)<\sigma(2i)$ for $1\leq i\leq n,$ and $\sigma(2j-1)<\sigma(2j+1)$ for $1\leq j\leq n-1.$

Pfaffian has the following Laplace expansion
\begin{align}\label{e:expansion}
\text{Pf(A)}=\sum\limits_{j=2}^{2n}(-1)^{j}a_{1j}\text{Pf($A_{1j}$)}.
\end{align}
More generally, fixing $i$, we have
\begin{align}\label{e:decomposition2}
\text{Pf(A)}=(-1)^{i-1}\sum\limits_{j\neq i}(-1)^{j}:a_{ij}:\text{Pf($A_{ij}$)}.
\end{align}
where $:a_{ij}:=a_{ij}$ if $i<j$ or $a_{ji}$ if $i>j$, and
$A_{ij}$ is the submatrix obtained from $A$ by deleting the $i$th, $j$th rows and $i$th, $j$th columns. We have supplied a proof of \eqref{e:decomposition2} in the appendix.

Suppose the operators $Q_i, Q_i^-$ satisfy the Clifford algebra relations \eqref{e:exchange}.
Then for general linear combinations $\omega_i$ of $Q_i$'s and $Q_i^-$'s:
\begin{align*}
\omega_i=\sum_{j\in\mathbb{Z}}v_{ij}Q_j+\sum_{j\in\mathbb{Z}}u_{ij}Q^-_j, \quad i=1,2,\cdots,m+n.
\end{align*}
The Wick formula \cite[Appendix, p.363, line 7-11]{DJKM} says that
\begin{align}\label{e:wick}
\langle \omega^{-}_{m}\omega^{-}_{m-1}\cdots \omega^{-}_{1}.1,  \omega_{m+1}\omega_{m+2}\cdots \omega_{m+n}.1\rangle_{-1}=
\begin{cases}
\text{Pf}(A)& \text{if $n+m$ is even}\\
0& \text{otherwise}
\end{cases}
\end{align}
where $A$ is the $(n+m)$ by $(n+m)$ antisymmetric matrix with entries $A_{ij}=
\langle \omega^{-}_{i}.1, \omega_{j}.1 \rangle_{-1}$, $1\leq i< j\leq n+m$.

Define $f_{j}$  ($j\in\mathbb{Z}$) by
\begin{align*}
f_{j}=\left\{\begin{array}{ll}
1 & \text {if $j=0$}\\
2 & \text {if $j>0$ and $s|j$}\\
0 & \text {otherwise}.
\end{array}\right.
\end{align*}
More generally, define ($m,n\geq0$)
\begin{align}
f_{(m,n)}=\left\{\begin{array}{ll}f_{m}f_{n}+2\sum_{j=1}^{n}(-1)^{j}f_{m+j}f_{n-j}& \text {if $n\geq1$}\\
f_m &\text {if $n=0$}.
\end{array}\right.
\end{align}

\begin{lem}\label{e:identity2}
Let $m,n\geq0$, $m+n\neq0$ and denote $r$ by $n\equiv r~ (mod~ s)$ with $0\leq r<s$, then\\
(1) $f_{(m,n)}=
\begin{cases}
4(-1)^r& \text {if $s\mid(n+m)$ and $r>0$}\\
2& \text {if $s\mid m$ and $n=0$}\\
-2& \text {if $s\mid n$ and $m=0$}\\
0& \text {otherwise}
\end{cases}$;\\
(2) $f_{(m,n)}=-f_{(n,m)}$;\\
(3) $f_{(m,n)}=f_{(n,m)}=0$ for $r=0$ and $m,n>0.$
\end{lem}
\begin{proof}
Both (2) and (3) follow directly from (1). So it is sufficient to prove (1). If $n=0$, then $f_{(m,n)}=f_{m}$ and (1) holds trivially. Suppose $n>0$ and $n=ts+r,$ $r=0,1,\ldots,s-1,$ then
\begin{align*}
f_{(m,n)}=
\begin{cases}
f_mf_n+2(-1)^nf_{m+n}f_0+2(-1)^{n-s}f_{m+n-s}f_{s}+\cdots+2(-1)^{n-ts}f_{m+n-ts}f_{ts}& \text {if $r>0$}\\
f_mf_n+2(-1)^nf_{m+n}f_0+2(-1)^{n-s}f_{m+n-s}f_{s}+\cdots+2(-1)^{n-(t-1)s}f_{m+n-(t-1)s}f_{(t-1)s}& \text {if $r=0$}.
\end{cases}
\end{align*}
So $f_{(m,n)}=0$ unless $s\mid(m+n).$ Now we suppose $s\mid(m+n)$, if $m=0$ then $n=st$, so $f_{(m,n)}=-2$. If $m>0$ it follows from a straightforward calculation that
$
f_{(m,n)}=
\begin{cases}
4(-1)^r& \text {if $r>0$}\\
0& \text {if $r=0$}.
\end{cases}
$
\end{proof}

Let $\mu=(\mu_1,\mu_2,\cdots,\mu_{a})\vdash n$, $\lambda=(\lambda_1,\lambda_2,\ldots, \lambda_b)\vdash n+sk$ be two strict partitions such that $\mu\subset \lambda$ and $a+b$ is even (where $\mu_a=0$ for $l(\mu)+l(\lambda)$ is odd). In order to get our result, we need to compute the coefficients of $Q_{\mu}.1$ that appear in $T^{(s)-}_{k}Q_{\lambda}.1$. Note that $Q(-z)=Q^{-}(z)$ so $(-1)^{m}Q_{-m}=Q^{-}_{m}$, which implies
$$\langle Q^{-}_{-\lambda}.1, Q_{\mu}.1\rangle_{-1}=(-1)^{|\la|}\langle Q_{\lambda}.1, Q_{\mu}.1\rangle_{-1}=\delta_{\lambda,\mu}(-1)^{|\mu|}2^{l(\mu)}.$$
On the other hand, Theorem \ref{t:Q} yields that
the coefficient $C(\lambda/\mu)$ of $Q_{\mu}.1$ in $T^{(s)-}_{k}Q_{\lambda}.1$ can be expressed as follows:
\begin{align*}
C(\lambda/\mu)=&2^{-l(\mu)}(-1)^{n}\langle  Q^{-}_{-\mu}.1,T^{(s)-}_{k}Q_{\lambda}.1 \rangle\\
=&2^{-l(\mu)}(-1)^{n}\left \langle Q^{-}_{-\mu}.1, \sum_{\nu\models k}2^{l(\nu)}Q_{\lambda-s\nu}.1\right \rangle_{-1}.
\end{align*}
Then
\begin{align*}
C(\lambda/\mu)=&
2^{-l(\mu)}(-1)^{n}\left \langle Q^{-}_{-\mu_{1}}\cdots Q^{-}_{-\mu_{a}}.1, \sum_{\mbox{\tiny$\begin{array}{c}
\nu_1,\nu_2,\cdots,\nu_{b}\geq0\\ \nu_1+\nu_2+\cdots+\nu_{b}=sk\end{array}$}}f_{\nu_1}\cdots f_{\nu_b}Q_{\lambda_1-\nu_1}\cdots Q_{\lambda_{b}-\nu_{b}}.1 \right \rangle_{-1}\\
=&2^{-l(\mu)}(-1)^{n}\left \langle Q^{-}_{-\mu_{1}}\cdots Q^{-}_{-\mu_{a}}.1, (\sum\limits_{\nu_1\geq0}f_{\nu_1}Q_{\lambda_1-\nu_1}) \cdots(\sum\limits_{\nu_b\geq0}f_{\nu_b}Q_{\lambda_b-\nu_b}).1\right \rangle_{-1}.
\end{align*}
Note that: $\langle Q^{-}_{-\mu_{1}}Q^{-}_{-\mu_2}\cdots Q^{-}_{-\mu_{a}}.1, Q_{a_1}Q_{a_2}\cdots Q_{a_m}.1 \rangle_{-1}=0$ unless $\sum a_i=|\mu|,$ where $a_i\in \mathbb{Z}.$

It follows from \eqref{e:wick} that
\begin{align*}
C(\lambda/\mu)=2^{-l(\mu)}(-1)^{n}\text{Pf}(M(\lambda/\mu))_{i,j=1,2,\ldots,a+b}.
\end{align*}
where $M(\lambda/\mu)$ is an $(a+b)$ by $(a+b)$ antisymmetric block matrix with entries
\begin{gather*}
\begin{pmatrix}M_{a\times a} & M_{a\times b} \\
 -M^{T}_{a\times b} & M_{b\times b}
\end{pmatrix}_{(a+b)\times (a+b)}
\end{gather*}
where $M_{a\times b}$, $M_{a\times a}$ and $M_{b\times b}$ are antisymmetric matrices.
By Proposition \ref{t:characterization} and the Wick formula, we have:\\
For $1\leq i<j\leq a$,
\begin{align*}
(M_{a\times a})_{i,j}&=\langle Q^{-}_{-\mu_{a-i+1}}.1, Q_{-\mu_{a-j+1}}.1\rangle_{-1}=0
\quad\quad\text{(since $Q_{-\mu_{a-j+1}}.1=0, 1<j\leq a$)}.
\end{align*}
For $1\leq i\leq a, 1\leq j\leq b$,
\begin{align*}
(M_{a\times b})_{i,j}&=\langle Q^{-}_{-\mu_{a-i+1}}.1, \sum\limits_{\nu_j\geq0}f_{\nu_j}Q_{\lambda_{j}-\nu_j}.1\rangle_{-1}\\
&=\begin{cases}
f_{\lambda_{j}}& \text{if $i=1$, and $\mu_{a}=0$;}\\
2(-1)^{\mu_{a-i+1}}f_{\lambda_{j}-\mu_{a-i+1}}& \text{otherwise}.
\end{cases}
\end{align*}
For $1\leq i<j\leq b$,
\begin{align*}
(M_{b\times b})_{i,j}&=\left \langle \sum\limits_{\nu_i\geq0}f_{\nu_i}Q^{-}_{\lambda_i-\nu_i}.1, \sum\limits_{\nu_j\geq0}f_{\nu_j}Q_{\lambda_j-\nu_j}.1\right \rangle_{-1}\\
&=f_{\lambda_i}f_{\lambda_j}+2\sum\limits_{\nu_j\geq0}^{\lambda_j-1}(-1)^{\lambda_j-\nu_j}f_{\lambda_i+\lambda_j-\nu_j}f_{\nu_j}\\
&=f_{(\lambda_i,\lambda_j)}.
\end{align*}

We introduce a modified matrix $\tilde{M}(\lambda/\mu)$ of $M(\lambda/\mu)$ as an $(a+b)$ by $(a+b)$ antisymmetric block matrix with entries
\begin{gather*}
\begin{pmatrix}\tilde{M}_{a\times a} & \tilde{M}_{a\times b} \\
 -\tilde {M}^{T}_{a\times b} & \tilde{M}_{b\times b}
\end{pmatrix}_{(a+b)\times (a+b)}
\end{gather*}
where $\tilde {M}_{a\times b}$, $\tilde{M}_{a\times a}$ and $\tilde{M}_{b\times b}$ are antisymmetric matrices, and
\begin{align*}
(\tilde{M}_{a\times a})_{i,j}&=0, \quad 1\leq i<j\leq a,\\
(\tilde{M}_{a\times b})_{i,j}&=f_{\lambda_{j}-\mu_{a-i+1}}, \quad 1\leq i\leq a, 1\leq j\leq b,\\
(\tilde{M}_{b\times b})_{i,j}&=f_{(\lambda_i,\lambda_j)}, \quad 1\leq i<j\leq b.
\end{align*}

Summarizing the above, we have shown the following.
\begin{thm}\label{t:Pf} Let $k$ be a positive integer and $\lambda\in SP_{n+sk}$, then we have
\begin{align}\label{e:Pf}
T^{(s)-}_{k}Q_{\lambda}.1=\sum_{\mu}\operatorname{Pf}(\tilde{M}(\lambda/\mu))Q_{\mu}.1
\end{align}
where the sum runs over all $\mu\in\mathcal{SP}_n$ such that $\mu\subset\lambda$.
\end{thm}
By duality, we have that
\begin{cor}\label{t:cor2} Let $k$ be a positive integer and $\mu\in SP_{n}$, then we have
\begin{align}
T^{(s)}_{k}Q_{\mu}.1=\sum_{\la}2^{l(\mu)-l(\la)}\operatorname{Pf}(\tilde{M}(\lambda/\mu))Q_{\la}.1
\end{align}
where the sum runs over all $\la\in\mathcal{SP}_{n+sk}$ such that $\mu\subset\lambda$.
\end{cor}

\subsection{\bf Properties and evaluation of $\operatorname{Pf}(\tilde{M}(\lambda/\mu))$.}
In this subsection, we will give some properties and formulae for $\operatorname{Pf}(\tilde{M}(\lambda/\mu))$ so that we can compute $\operatorname{Pf}(\tilde{M}(\lambda/\mu))$ more efficiently. Based on these, we can give a combinatorial Murnaghan-Nakayama rule. We always assume $\mu=(\mu_1,\mu_2,\cdots,\mu_a)\subset\lambda=(\la_1,\la_2,\cdots,\la_b)$ such that $a+b$ is even unless specified otherwise (which is always possible by appending a zero to $\mu$). In other words, $a=l(\mu)$ or $l(\mu)+1$ and $b=l(\la)$.
We write ${\rm Pf}(\tilde{M}(\lambda))$ for short when $\mu=\emptyset$.
Note that this implicitly implies that $|\lambda|-|\mu|\equiv 0~ (mod~ s)$.

We introduce some notations. Let $\lambda$ be a strict partition. Denote
\begin{align*}
N_r(\lambda)&=\{\lambda_i\mid \lambda_i\equiv r ({\rm mod}~ s)\}, \quad n_r(\lambda)={\rm Card}N_{r}(\lambda), \quad r=0,1,2,\cdots,s-1.\\
\lambda^{\hat{i}}&=(\lambda_1,\lambda_2,\cdots,\lambda_{i-1},\lambda_{i+1},\cdots), \quad \lambda^{[i]}=(\lambda_{i+1},\lambda_{i+2},\cdots).
\end{align*}
Arrange the elements of $N_r(\lambda)$ in the descending order:
\begin{align*}
N_r(\lambda)_{>}=\{\lambda_{r,1}>\la_{r,2}>\cdots>\la_{r,n_r(\lambda)}\}.
\end{align*}
Then $N_r(\la)_>$ is also a strict partition and a subpartition of $\lambda$. It is clear that $\lambda$ can be written as the non-intersecting union of these subpartitions.
\begin{align*}
\lambda=\biguplus_{r=0}^{s-1}N_r(\lambda)_{>}.
\end{align*}

\begin{lem}\label{t:Pfit}
Fix $1\leq j\leq a$. If $\mu_j\in N_r(\mu)$, then we have the following iterative formula for ${\rm Pf}(\tilde{M}(\lambda/\mu)).$
\begin{align}\label{e:Pfit}
{\rm Pf}(\tilde{M}(\lambda/\mu))=\sum\limits_{\lambda_i\in N_r(\la),\lambda_i\geq\mu_j}(-1)^{i-j}f_{\lambda_i-\mu_j}{\rm Pf}(\tilde{M}(\lambda^{\hat{i}}/\mu^{\hat{j}})).
\end{align}
\end{lem}
\begin{proof} By the above definition of $\tilde{M}(\lambda/\mu)$ the top-left block matrix $\tilde{M}_{a\times a}=0.$ Therefore, \eqref{e:Pfit} is obtained by taking the $(a+1-j)$th row decomposition of ${\rm Pf}(\tilde{M}(\lambda/\mu))$ using \eqref{e:decomposition2}.
\end{proof}


Next we will give some necessary conditions for ${\rm Pf}(\tilde{M}(\lambda/\mu))\neq0.$
\begin{prop}\label{t:con1}
If there exists $0\leq r_0\leq s-1$ such that the cardinality $n_{r_0}(\mu)$ obeys
$n_{r_0}(\mu)+1<n_{r_0}(\la)$ or $n_{r_0}(\lambda)<n_{r_0}(\mu)$, then ${\rm Pf}(\tilde{M}(\lambda/\mu))=0.$
\end{prop}
\begin{proof}
If $n_{r_0}(\mu)+1<n_{r_0}(\la)$, we argue by induction on $a+n_{r_0}(\la)$. The initial step is $a+n_{r_0}(\la)=2,$ in this case $a=0$ and $n_{r_0}(\la)=2$. Suppose $\lambda_i\equiv\la_j\equiv r_0 ({\rm mod}~s)$. By Lemma \ref{e:identity2} the $i$th column and the $j$th column are equal in $\tilde{M}_{b\times b}(\lambda)$. Thus ${\rm Pf}(\tilde{M}(\lambda))=0$.

Now assume it holds for $<a+n_{r_0}(\la)$.
(i) If $n_{r_0}(\mu)=0$,
then by Lemma \ref{e:identity2}, the $(a+i)$th column is also identical to
the $(a+j)$th column in $\tilde{M}(\lambda/\mu)$. So ${\rm Pf}(\tilde{M}(\lambda/\mu))=0.$ (ii) If $n_{r_0}(\mu)>0$ and
$\mu_k\in N_{r_0}(\mu)$. It follows from Lemma \ref{t:Pfit} (taking $\mu_k$ as $\mu_j$ in the Lemma) and the inductive hypothesis that ${\rm Pf}(\tilde{M}(\lambda/\mu))=0$. The case for $n_{r_0}(\lambda)<n_{r_0}(\mu)$ can be proved similarly by induction on $n_{r_0}(\mu)$.
\end{proof}

\begin{prop}\label{t:con4} If $n_{0}(\lambda)\neq n_{0}(\mu)$, then ${\rm Pf}(\tilde{M}(\lambda/\mu))=0.$
\begin{proof}
By Proposition \ref{t:con1}, we have ${\rm Pf}(\tilde{M}(\lambda/\mu))=0$ unless $0\leq n_{0}(\lambda)-n_{0}(\mu)\leq 1$. So it is enough to consider the case of $n_{0}(\lambda)= n_{0}(\mu)+1$. We argue by induction on $n_{0}(\mu)$. If $n_{0}(\mu)=0$, then $n_{0}(\lambda)=1$. Suppose $\lambda_{i}\in N_{0}(\lambda)$, then by the definition of $f_{j}$ and Lemma \ref{e:identity2}, the elements in the $(a+i)$th column of $\tilde{M}(\la/\mu)$ are all zero. Thus, ${\rm Pf}(\tilde{M}(\lambda/\mu))=0.$ Assume it holds for $<n_{0}(\mu)$. Now for $\mu_k\in N_{0}(\mu)$, it follows from Lemma \ref{t:Pfit} (taking $\mu_k$ as $\mu_j$ in the Lemma) and the inductive hypothesis that ${\rm Pf}(\tilde{M}(\lambda/\mu))=0$.
\end{proof}
\end{prop}

\begin{lem}\label{t:con2'}
Suppose $a=0$. If there exists $1\leq r_0\leq s-1$ such that $n_{r_0}(\lambda)\neq n_{s-r_0}(\lambda)$, then ${\rm Pf}(\tilde{M}(\lambda))=0.$
\end{lem}
\begin{proof}
Without loss of generality, we assume $n_{r_0}(\lambda)> n_{s-r_0}(\lambda)$. The lemma is proved by induction on $n_{r_0}(\la)$. Just note that for any fixed $i$ the following iterative formula holds:
\begin{align*}
{\rm Pf}(\tilde{M}(\lambda))=\sum_{s\mid (\la_i+\la_j)}(-1)^{i+j-1}:f_{(\la_i,\la_j)}:
{\rm Pf}(\tilde{M}(\lambda^{\hat{i},\hat{j}}))
\end{align*}
where $\lambda^{\hat{i},\hat{j}}=(,\cdots,\lambda_{i-1},\lambda_{i+1},\cdots,\lambda_{j-1},\lambda_{j+1},\cdots).$
\end{proof}
More generally, we have:
\begin{prop}\label{t:con2}
If there exists $1\leq r_0\leq s-1$ such that $n_{r_0}(\lambda)-n_{r_0}(\mu)\neq n_{s-r_0}(\lambda)-n_{s-r_0}(\mu)$, then ${\rm Pf}(\tilde{M}(\lambda/\mu))=0.$
\end{prop}
\begin{proof} Without loss of generality, we assume $n_{r_0}(\lambda)-n_{r_0}(\mu)> n_{s-r_0}(\lambda)-n_{s-r_0}(\mu)$. We prove the proposition
by induction on $a$. The initial step is Lemma \ref{t:con2'}.

Now assume it holds for $<a$. By Proposition \ref{t:con1}, we have ${\rm Pf}(\tilde{M}(\lambda/\mu))=0$ unless $0\leq n_{r_0}(\lambda)-n_{r_0}(\mu)\leq1$ and $0\leq n_{s-r_0}(\lambda)-n_{s-r_0}(\mu)\leq 1$. So it is enough to consider the case of $1\geq n_{r_0}(\lambda)-n_{r_0}(\mu)>n_{s-r_0}(\lambda)-n_{s-r_0}(\mu)\geq0$. (i) If $n_{r_0}(\mu)=0$, then $n_{r_0}(\lambda)=1$, $n_{s-r_0}(\la)=n_{s-r_0}(\mu)$. Suppose $\la_i\in N_{r_0}(\la)$ and $n_{s-r_0}(\la)=n_{s-r_0}(\mu)=p\geq0$. If $p=0$, then by Lemma \ref{e:identity2}, the elements in the $(a+i)$th column in $\tilde{M}(\la/\mu)$ are all zero, so ${\rm Pf}(\tilde{M}(\lambda/\mu))=0$. If $p>0$, this implies $N_{s-r_0}(\mu)\neq\emptyset$. Suppose $\mu_j\in N_{s-r_0}(\mu)$. Then we can use induction on $p$, by Lemma \ref{t:Pfit} and the inductive assumption we have ${\rm Pf}(\tilde{M}(\lambda/\mu))=0$. (ii) If $n_{r_0}(\mu)>0$ and $\mu_k\in N_{r_0}(\mu)$. It follows from Lemma \ref{t:Pfit} (regarding $\mu_k$ as $\mu_j$ in the Lemma) and the inductive hypothesis that ${\rm Pf}(\tilde{M}(\lambda/\mu))=0.$
\end{proof}

\begin{lem}\label{l:con3}
Suppose $0\leq r_0\leq s-1$ and $N_{r_0}(\mu)\neq\emptyset$, if $\mu_{r_0,1}>\la_{r_0,1}$ or $\la_{r_0,2}>\mu_{r_0,1}$, then $\operatorname{Pf}(\tilde{M}(\lambda/\mu))=0$.
\end{lem}
\begin{proof} Suppose $\mu_{r_0,1}=\mu_{k}$. If $\mu_{r_0,1}>\la_{r_0,1}$, then the $(a+1-k)$th row of $\tilde{M}(\lambda/\mu)$ is equal to $0$. Thus, $\operatorname{Pf}(\tilde{M}(\lambda/\mu))=0$. Now consider the case $\la_{r_0,2}>\mu_{r_0,1}$. Suppose $\la_{r_0,1}$ ($\la_{r_0,2}$) occupies the $i$th ($j$th) entry in $\la$.  Lemma \ref{e:identity2} implies that the $(a+i)$th column coincides with the $(a+j)$th column
in $\tilde{M}(\lambda/\mu)$. So ${\rm Pf}(\tilde{M}(\lambda/\mu))=0.$
\end{proof}

\begin{prop}\label{p:con3}
Suppose $0\leq r_0\leq s-1$ and $N_{r_0}(\mu)\neq\emptyset$. If there exists $j_0$ such that $\mu_{r_0,j_0}>\la_{r_0,j_0}$ or $\la_{r_0,j_0+1}>\mu_{r_0,j_0}$, then $\operatorname{Pf}(\tilde{M}(\lambda/\mu))=0$.
\end{prop}
\begin{proof}
We use induction on $j_0$. When $j_0=1$, it is Lemma \ref{l:con3}. Suppose it holds for any $j_0<m$. Then $\operatorname{Pf}(\tilde{M}(\lambda/\mu))=0$ unless $\la_{r_0,1}\geq\mu_{r_0,1}\geq\cdots\geq\la_{r_0,m-1}\geq\mu_{r_0,m-1}\geq\la_{r_0,m}$. If $\mu_{r_0,m}>\la_{r_0,m}$(resp. $\la_{r_0,m+1}>\mu_{r_0,m}$), then by Lemma \ref{t:Pfit} (regarding $\mu_{b(m)}$ as $\mu_j$)
\begin{align*}
\operatorname{Pf}(\tilde{M}(\lambda/\mu))=\sum_{i=1}^{m-1}(-1)^{a(i)-b(m)}f_{\la_{a(i)}-\mu_{b(m)}}\operatorname{Pf}(\tilde{M}((\la_1,\cdots,\hat{\la}_{a(i)},\cdots)/\mu^{\hat{b}(m)}))
\end{align*}
resp. (regarding $\mu_{b(1)}$ as $\mu_j$)
\begin{align*}
\operatorname{Pf}(\tilde{M}(\lambda/\mu))&=(-1)^{a(1)-b(1)}f_{\la_{r_0,1}-\mu_{r_0,1}}\operatorname{Pf}(\tilde{M}((\la_1,\cdots,\hat{\la}_{a(1)},\cdots)/
\mu^{\hat{b}(1)}))\\
&+(-1)^{a(2)-b(1)}f_{\la_{r_0,2}-\mu_{r_0,1}}\operatorname{Pf}(\tilde{M}((\la_1,\cdots,\hat{\la}_{a(2)},\cdots)/
\mu^{\hat{b}(1)}))
\end{align*}
where $a(i)$(resp. $b(i)$) denotes the entry occupied by $\la_{r_0,i}$(resp. $\mu_{r_0,i}$) and $(\la_1,\cdots,\hat{\la}_{a(i)},\cdots)$ means omitting the part $\la_{a(i)}$ in $\la$. By the inductive hypothesis, each summand above is zero. So $\operatorname{Pf}(\tilde{M}(\lambda/\mu))=0$.
\end{proof}

Let $\mathcal{SP}_{s,k}$ be the set of strict partition pairs $(\lambda,\mu)$ such that $\mu\subset\lambda\vdash|\mu|+sk$, $l(\lambda)+l(\mu)$ is even, and
in addition \\
$(\rmnum{1})$  $n_{0}(\lambda)=n_{0}(\mu)$ and $0\leq n_{r}(\lambda)-n_{r}(\mu)=n_{s-r}(\lambda)-n_{s-r}(\mu)\leq1$ for any $0< r\leq s-1$;\\
$(\rmnum{2})$ For any $N_{r}(\mu)\neq\emptyset$ and $1\leq j_0\leq n_r(\mu)$, $\la_{r,j_0}\geq\mu_{r,j_0}\geq\la_{r,j_0+1}$.\\

Let $\lambda, \mu$ be strict partitions such that $\mu\subset\lambda$, a skew diagram $\lambda/\mu$ is called a vertical (resp. horizontal) {\it strip} if each row (resp. column) contains at most one box \cite[Ch. I, p.5]{Mac}. If $\lambda/\mu$ is a horizontal strip, denote by $a(\lambda/\mu)$ the {\it $a$-number} of the skew diagram $\lambda/\mu$, which is the number of integers $i\geq1$ such that $\lambda/\mu$ has a box in the $i$th column but not in the $(i+1)$th column. The Pieri rule for Schur's $Q$-functions states that \cite[Ch. III, p.255]{Mac}
\begin{align}\label{e:Pieri}
q_kQ_{\mu}=\sum\limits_{\lambda}2^{a(\lambda/\mu)}2^{l(\mu)-l(\lambda)}Q_{\lambda}
\end{align}
where the sum runs over all strict partitions $\lambda\supset\mu$ such that $\lambda/\mu$ is a horizontal $k$-strip.

We remark that $\mathcal{SP}_{s,k}$ can be equivalently characterized as follows.
A pair $(\la,\mu)\in \mathcal{SP}_{s,k}$ if and only if
the following symmetry and non-intersecting properties hold\\
$(\rmnum{1}^{'})$  $n_{0}(\lambda)=n_{0}(\mu)$ and $0\leq n_{r}(\lambda)-n_{r}(\mu)=n_{s-r}(\lambda)-n_{s-r}(\mu)\leq1$ for any $0< r\leq s-1$;\\
$(\rmnum{2}^{'})$  $N_{r}(\la)_{>}/N_{r}(\mu)_{>}$ is a horizontal strip  for any $0\leq r\leq s-1$.\\

For $(\la,\mu)\in \mathcal{SP}_{s,k}$, we define the {\it normalized $a$-number} of skew diagram $\lambda/\mu$
\begin{equation}\label{e:A}
A(\la/\mu)=\sum_{r=0}^{s-1}a(N_{r}(\la)_{>}/N_{r}(\mu)_{>}).
\end{equation}
Note that if $(\lambda, \mu)\in \mathcal{SP}_{s,k}$ then it can be regarded as a non-intersecting union of horizontal strips.
So we can call $\la/\mu$ a symmetric horizontal $(s,k)$-strip if $(\la,\mu)\in \mathcal{SP}_{s,k}$.

Propositions \ref{t:con1}, \ref{t:con4}, \ref{t:con2} and \ref{p:con3} imply that ${\rm Pf}(\tilde{M}(\lambda/\mu))=0$ unless $(\la,\mu)\in\mathcal{SP}_{s,k}.$ Next we prove that $(\la,\mu)\in\mathcal{SP}_{s,k}$ is also a sufficient condition for ${\rm Pf}(\tilde{M}(\lambda/\mu))\neq0$. Moreover we will give
a generalization of \eqref{e:Pieri}:
\begin{align}
(p_s\circ q_k)Q_{\mu}=\sum\limits_{\lambda}\epsilon(\la/\mu)2^{A(\lambda/\mu)}2^{l(\mu)-l(\lambda)}Q_{\lambda}
\end{align}
summed over all strict partitions $\lambda\supset\mu$ such that $\lambda/\mu$ is a symmetric horizontal $(s,k)$-strip. Here $\epsilon(\la/\mu)=\pm1$.

Let us begin by introducing two operations on matrices. Let $B=(b_{ij})_{m\times m}$ be a matrix, we define \\
$(\rmnum{1})$ The column rotation $C_{p}^{q}$ of matrix $B$: columns $p, p+1, \ldots, q-1, q$ are cyclically rotated to columns $q, p, p+1, \ldots, q-1$.\\
$(\rmnum{2})$ The row rotation $R_{p}^q$ of matrix $B$: rows $p, p+1, \ldots, q-1, q$ are cyclically rotated to rows $q, p, p+1, \ldots, q-1$.

For example, suppose $B=(b_{ij})_{4\times 4}$, then
\begin{align*}
C_{1}^3(B)=\left(\begin{array}{cccc}
b_{12}&b_{13}&b_{11}&b_{14}\\
b_{22}&b_{23}&b_{21}&b_{24}\\
b_{32}&b_{33}&b_{31}&b_{34}\\
b_{42}&b_{43}&b_{41}&b_{44}\\
\end{array}\right),~~~
R_{1}^3(B)=\left(\begin{array}{cccc}
b_{21}&b_{22}&b_{23}&b_{24}\\
b_{31}&b_{32}&b_{33}&b_{34}\\
b_{11}&b_{12}&b_{13}&b_{14}\\
b_{41}&b_{42}&b_{43}&b_{44}\\
\end{array}\right).
\end{align*}
Clearly the two operators commute: $C^{i}_{j}R_{p}^{q}=R^{q}_{p}C^{i}_{j}$. If $B$ is an antisymmetric matrix, then $R^{q}_{p}C^{q}_{p}(B)$ is also an antisymmetric matrix and ${\rm Pf}(R^{q}_{p}C^{q}_{p}(B))=(-1)^{q-p}{\rm Pf}(B)$.
\begin{lem}\label{l:RC}
Let $0\leq r_0\leq s-1$, $\mu_1\in N_{r_0}(\mu)$, $\la_{r_0,1}>\la_{r_0,2}=\mu_{r_0,1}$, $\la_{r_0,1}$(resp. $\la_{r_0,2}$) occupies the $i$-th(resp. $j$-th) entry in $\la$. Then ${\rm Pf}(\tilde{M}(\la^{\hat{i}}/\mu^{\hat{1}}))=(-1)^{j-i-1}{\rm Pf}(\tilde{M}(\la^{\hat{j}}/\mu^{\hat{1}}))$.
\end{lem}
\begin{proof}
It follows from the fact that $R^{a+j-2}_{a+i-1}C^{a+j-2}_{a+i-1}(\tilde{M}(\la^{\hat{j}}/\mu^{\hat{1}}))=\tilde{M}(\la^{\hat{i}}/\mu^{\hat{1}})$.
\end{proof}

In the next theorem $\epsilon(\lambda/\mu)$ is determined implicitly in the proof.  An explicit formula of $\epsilon(\lambda/\mu)$ will be
given later in Theorem \ref{t:thm1}.
\begin{thm}\label{t:iff}
If $(\la,\mu)\in \mathcal{SP}_{s,k}$, then ${\rm Pf}(\tilde{M}(\la/\mu))=\epsilon(\la/\mu)2^{A(\la/\mu)}$, for some $\epsilon(\la/\mu)$, where $A(\lambda/\mu)$ is the normalized $a$-number of
$\lambda/\mu$ and $\epsilon(\la/\mu)=\pm1$. In particular ${\rm Pf}(\tilde{M}(\la/\mu))\neq0$.
\end{thm}
\begin{proof} Write $\mu=(\mu_1, \ldots, \mu_a)$ and we use induction on $a=l(\mu)$ (or $l(\mu)+1$ if $0$ is appended to $\mu$). 
If $a=0$, then $b=l(\lambda)$ is even. The partition pair $(\la,\emptyset)\in \mathcal{SP}_{s,k}$ implies that
$n_0(\la)=0$, 
$n_r(\la)=0$ or $1$ $(1\leq r\leq s-1)$, and if $n_{r}(\la)=1$ then $n_{s-r}(\la)=1$.
Suppose $\{r\mid n_r(\la)=1\}_{<}=\{r_1,r_2,\cdots,r_b\}$, then $r_j+r_{b+1-j}=s$, $1\leq j\leq\frac{b}{2}$. Assume $\la_{i_j}\in N_{r_j}(\la)$, $j=1,2,\cdots,b$. Then the $b\times b$ matrix $\tilde{M}(\la)$ has $(-1)^{r_{b+1-j}}4$ at $(i_j,i_{b+1-j})$-entry and $0$ elsewhere, i.e. each column and each row
contains exactly one nonzero entry $4$ (or $-4$). So ${\rm Pf}(\tilde{M}(\la))=\pm2^b$.

Assume the theorem holds for any pair of partitions $(\bar{\la},\bar{\mu})\in \mathcal{SP}_{s,k}$ with $\bar{a}<a$ (here $\bar{a}=l(\bar{\mu})$ or $l(\bar{\mu})+1$ depending on $(\bar{\la},\bar{\mu})$). Consider the case $(\la,\mu)\in\mathcal{SP}_{s,k}$ with $a=l(\mu)$ or $l(\mu)+1$. Suppose $\mu_1\in N_{r_0}(\mu)$, then $\la_{r_0,1}\geq\mu_1\geq\la_{r_0,2}$. $\la_{r_0,1}$(resp. $\la_{r_0,2}$) occupies the $i$th(resp. $j$th) entry in $\la$.
We now divide the proof into three cases, where each case uses Lemma \ref{t:Pfit} (regarding $\mu_1$ as $\mu_j$ in the Lemma).

{\bf Case 1:} If $\la_{r_0,1}=\mu_1>\la_{r_0,2}$, we have ${\rm Pf}(\tilde{M}(\la/\mu))=(-1)^{i-1}{\rm Pf}(\tilde{M}(\la^{\hat{i}}/\mu^{\hat{1}}))$. In this case $$a(N_{r_0}(\la^{\hat{i}})_{>}/N_{r_0}(\mu^{\hat{1}})_{>})=a(N_{r_0}(\la)_{>}/N_{r_0}(\mu)_{>}).$$
Then $A(\la^{\hat{i}}/\mu^{\hat{1}})=A(\la/\mu)$. By the inductive hypothesis
\begin{align*}
{\rm Pf}(\tilde{M}(\la/\mu))&=(-1)^{i-1}{\rm Pf}(\tilde{M}(\la^{\hat{i}}/\mu^{\hat{1}}))\\
&=(-1)^{i-1}\epsilon(\la^{\hat{i}}/\mu^{\hat{1}})2^{A(\la^{\hat{i}}/\mu^{\hat{1}})}\\
&=(-1)^{i-1}\epsilon(\la^{\hat{i}}/\mu^{\hat{1}})2^{A(\la/\mu)}.
\end{align*}
{\bf Case 2:} If $\la_{r_0,1}>\mu_1>\la_{r_0,2}$, then ${\rm Pf}(\tilde{M}(\la/\mu))=(-1)^{i-1}2{\rm Pf}(\tilde{M}(\la^{\hat{i}}/\mu^{\hat{1}})).$ Note that $$a(N_{r_0}(\la^{\hat{i}})_{>}/N_{r_0}(\mu^{\hat{1}})_{>})=a(N_{r_0}(\la)_{>}/N_{r_0}(\mu)_{>})-1.$$
Then $A(\la^{\hat{i}}/\mu^{\hat{1}})=A(\la/\mu)-1$. By the inductive hypothesis again,\begin{align*} {\rm Pf}(\tilde{M}(\la/\mu))&=(-1)^{i-1}2{\rm Pf}(\tilde{M}(\la^{\hat{i}}/\mu^{\hat{1}}))\\
&=(-1)^{i-1}2\epsilon(\la^{\hat{i}}/\mu^{\hat{1}})\times2^{A(\la^{\hat{i}}/\mu^{\hat{1}})}\\
&=(-1)^{i-1}\epsilon(\la^{\hat{i}}/\mu^{\hat{1}})2^{A(\la/\mu)}.
\end{align*}
{\bf Case 3:} If $\la_{r_0,1}>\mu_1=\la_{r_0,2}$, then $a(N_{r_0}(\la^{\hat{i}})_{>}/N_{r_0}(\mu^{\hat{1}})_{>})=a(N_{r_0}(\la)_{>}/N_{r_0}(\mu)_{>})$.\\
 So $A(\la^{\hat{i}}/\mu^{\hat{1}})=A(\la/\mu)$. Therefore
\begin{align*}
{\rm Pf}(\tilde{M}(\la/\mu))=&(-1)^{i-1}2{\rm Pf}(\tilde{M}(\la^{\hat{i}}/\mu^{\hat{1}}))+(-1)^{j-1}{\rm Pf}(\tilde{M}(\la^{\hat{j}}/\mu^{\hat{1}}))\\
=&(-1)^{i-1}{\rm Pf}(\tilde{M}(\la^{\hat{i}}/\mu^{\hat{1}}))~~~~(\text {by Lemma \ref{l:RC}})\\
=&(-1)^{i-1}\epsilon(\la^{\hat{i}}/\mu^{\hat{1}})2^{A(\la^{\hat{i}}/\mu^{\hat{1}})}~~~~(\text {by induction})\\
=&(-1)^{i-1}\epsilon(\la^{\hat{i}}/\mu^{\hat{1}})2^{A(\la/\mu)}.
\end{align*}
Summarizing the three cases, ${\rm Pf}(\tilde{M}(\la/\mu))=\epsilon(\la/\mu)2^{A(\la/\mu)}$, where $\epsilon(\la/\mu)=(-1)^{i-1}\epsilon(\la^{\hat{i}}/\mu^{\hat{1}})$.
\end{proof}

\begin{rem} When $\mu=\emptyset$, the proof of Theorem \ref{t:iff} implies that
 ${\rm Pf}(\tilde{M}(\la))=\pm2^{l(\la)}$ for $(\la,\emptyset)\in \mathcal{SP}_{s,k}$. Therefore
\begin{align}
p_{s}\circ q_{k}=\sum_{(\la,\emptyset)\in\mathcal{SP}_{s,k}}\pm Q_{\lambda},
\end{align}
which means that the plethysm $p_{s}\circ q_{k}$ is multiplicity-free.
\end{rem}

Before closing this section, we determine $\epsilon(\la/\mu)$. Let $\la=(\la_1,\la_2,\cdots,\la_b)$, $\mu=(\mu_1,\mu_2,\cdots,\mu_a)$ for $(\la,\mu)\in\mathcal{SP}_{s,k}$. If parts of $\lambda$ are permuted by $\sigma\in\mathfrak{S}_{b}$, the resulted composition is denoted as $\sigma(\lambda)$.
   We now reordered $\lambda$ as $\tilde{\la}=\sigma(\la)=(\tilde{\la}_1, \ldots, \tilde{\la}_b)$ according to the following procedure:

(1) when $1\leq i\leq a$, if $\mu_i=\mu_{r_i,j_i}$, then $\tilde{\la}_i=\la_{r_i,j_i}$, denote $F(\la)=(\tilde{\la}_1,\tilde{\la}_2,\cdots,\tilde{\la}_a)$.

(2) when $a<i\leq b$, denote $R(\la)=\la\setminus F(\la)$. Since $(\la,\mu)\in\mathcal{SP}_{s,k}$, both $N_{j}(R(\la))$ and $N_{s-j}(R(\la))$ have only one element or are empty ($j=1,3,5,\cdots,s-2$). Suppose those pairs of nonempty sets are $(N_{j_1}(R(\la)),N_{s-j_1}(R(\la)))\cdots (N_{j_k}(R(\la)),N_{s-j_k}(R(\la)))$, where $j_1<j_2<\cdots<j_k$ are all odd. Then we define $\tilde{\la}_i\in N_{j_{\frac{i-a+1}{2}}}(R(\la))$ and $\tilde{\la}_{i+1}\in N_{s-j_{\frac{i-a+1}{2}}}(R(\la))$, $i=a+1,a+3,\cdots,b-1.$  (see Example \ref{t:sgn})

Note that $\tilde{M}(\lambda/\mu)$ also makes sense even when $\la$ and $\mu$ are compositions.
\begin{lem}\label{l:permutation}
Let $\sigma\in\mathfrak{S}_{b}$ be the permutation defined above, then ${\rm Pf}(\tilde{M}(\la/\mu))=sgn(\sigma){\rm Pf}(\tilde{M}(\sigma(\la)/\mu))$.
\end{lem}
\begin{proof}
Note that $\tilde{M}((\cdots,\la_i,\cdots,\la_j,\cdots)/\mu)$ can be obtained from $\tilde{M}((\cdots,\la_j,\cdots,\la_i,\cdots)/\mu)$ by exchanging the $(a+i)$th row with $(a+j)$th row and the $(a+i)$th column and the $(a+j)$th column.
\end{proof}

\begin{thm}\label{t:thm1}
Suppose $(\la,\mu)\in\mathcal{SP}_{s,k}$, and let $\sigma$ be the permutation sending $\la$ to $\tilde{\la}$ with respect to $\mu$, then $\epsilon(\la/\mu)=sgn(\sigma)$.
\end{thm}
\begin{proof} By Theorem \ref{t:iff} and Lemma \ref{l:permutation}, it is enough
 to prove ${\rm Pf}(\tilde{M}(\sigma(\la)/\mu))>0$. This is verified by induction on $a=l(\mu)$ (or $l(\mu)+1$). If $a=0$, then
 $b=l(\la)$ is even. By Lemma \ref{e:identity2}, we have $f_{(\tilde{\la}_{i},\tilde{\la}_{i+1})}=4$ if $i=1,3,\cdots,b-1,$ and $f_{(\tilde{\la}_{i},\tilde{\la}_{j})}=0$ else if $i<j$.
This means
 \begin{align*}
\tilde{M}(\sigma(\la))=\tilde{M}(\tilde{\la})=\left(\begin{array}{ccccccc}
0&4&0&0&\cdots&0&0\\
-4&0&0&0&\cdots&0&0\\
0&0&0&4&\cdots&0&0\\
0&0&-4&0&\cdots&0&0\\
\vdots&\vdots&\vdots&\vdots&\ddots&\vdots&\vdots\\
0&0&0&0&\cdots&0&4\\
0&0&0&0&\cdots&-4&0\\
\end{array}\right)_{b\times b},\quad{\rm Pf}(\tilde{M}(\sigma(\la)))=2^b.
\end{align*}
The rest of the proof is similar to that of Theorem \ref{t:iff}. Just note $i=1$ for $\sigma(\la)$ in this case.
\end{proof}
\begin{exmp}\label{t:sgn}
Let $s=7$, $k=7$, $\la=(24_3,23_2,20_6,18_4,17_3,16_2,6_6,5_5,1_1)$ and $\mu=(23_2,18_4,17_3,13_6,10_3)$, where the subscripts record the remainders modulo
 $s$. We can check $(\la,\mu)\in\mathcal{SP}_{s,k}$. Then
\begin{align*}
\begin{cases}
\mu_1=23=\mu_{2,1}\longleftrightarrow \la_{2,1}=\tilde{\lambda}_{1}=\lambda_{\sigma(1)}=23=\lambda_{2},& \\
\mu_2=18=\mu_{4,1}\longleftrightarrow \la_{4,1}=\tilde{\lambda}_{2}=\lambda_{\sigma(2)}=18=\lambda_{4},&\\
\mu_3=17=\mu_{3,1}\longleftrightarrow \la_{3,1}=\tilde{\lambda}_{3}=\lambda_{\sigma(3)}=24=\lambda_{1},& \\
\mu_{4}=13=\mu_{6,1}\longleftrightarrow \la_{6,1}=\tilde{\lambda}_{4}=\lambda_{\sigma(4)}=20=\lambda_{3},&\\
\mu_5=10=\mu_{3,2}\longleftrightarrow \la_{3,2}=\tilde{\lambda}_{5}=\lambda_{\sigma(5)}=17=\lambda_{5}.&
\end{cases}
\end{align*}
So $F(\la)=(23,18,24,20,17)$ and $R(\la)=(16,6,5,1)$. Since $1\in N_{1}(R(\la))$, $6\in N_{6}(R(\la))$, $5\in N_{5}(R(\la))$, $16\in N_{2}(R(\la))$,
the nonempty pairs are $(\{1\},\{6\})$ and $(\{5\},\{16\})$. Thus
$$\tilde{\lambda}_{6}=\lambda_{\sigma(6)}=1=\lambda_{9},~\tilde{\lambda}_{7}=\lambda_{\sigma(7)}=6=\lambda_{7},~ \tilde{\lambda}_{8}=\lambda_{\sigma(8)}=5=\lambda_{8},~ \tilde{\lambda}_{9}=\lambda_{\sigma(9)}=16=\lambda_{6}.$$
Therefore $\tilde{\la}=\sigma(\lambda)=(23,18,24,20,17,1,6,5,16)$, and $\sigma=(1243)(69).$
\end{exmp}

\begin{rem} We give a graphical algorithm to compute $sgn(\sigma)$ as follows:\\
(1) write $\la$ in the top line and $\mu$ in the bottom line;\\
(2) connect top part $\mu_i$ with the lower part $\la_j$ if their congruents are equal modulo $s$, say, $\mu_{r,i}$ is connected with $\la_{r,i}$, $0\leq r\leq s-1$, $i\geq1$;\\
(3) copy the unconnected parts of $\la$ to the bottom line according to the order $\tilde{\la}_{a+1}, \tilde{\la}_{a+2},\cdots,\tilde{\la}_{b}$;\\
(4) connect the unconnected parts with themselves.\\
Then the inversion number of $\sigma$ equals to the number of intersection points except vertical lines.
The following graph is for Example \ref{t:sgn}.
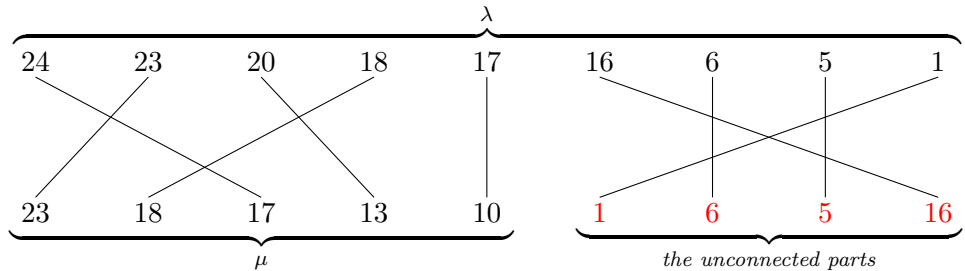
\begin{gather*}
\begin{tikzpicture}[scale=1]
    \coordinate (Origin)   at (0,0);
    \coordinate (XAxisMin) at (0,0);
    \coordinate (XAxisMax) at (9,0);
    \coordinate (YAxisMin) at (0,0);
    \coordinate (YAxisMax) at (0,-9);
    \draw [thin, black] (0,-0.2) -- (3,-1.8);
    \draw [thin, black] (1.5,-0.2) -- (0,-1.8);
    \draw [thin, black] (4.5,-0.2) -- (1.5,-1.8);
    \draw [thin, black] (3,-0.2) -- (4.5,-1.8);
    \draw [thin, black] (6,-0.2) -- (6,-1.8);
    \draw [thin, black] (7.5,-0.2) -- (12,-1.8);
    \draw [thin, black] (9,-0.2) -- (9,-1.8);
    \draw [thin, black] (10.5,-0.2) -- (10.5,-1.8);
     \draw [thin, black] (12,-0.2) -- (7.5,-1.8);
      \node[inner sep=2pt] at (0,0) {$24$};
     \node[inner sep=2pt] at (1.5,0) {$23$};
     \node[inner sep=2pt] at (3,0) {$20$};
     \node[inner sep=2pt] at (4.5,0) {$18$};
     \node[inner sep=2pt] at (6,0) {$17$};
     \node[inner sep=2pt] at (7.5,0) {$16$};
    \node[inner sep=2pt] at (9,0) {$6$};
    \node[inner sep=2pt] at (10.5,0) {$5$};
     \node[inner sep=2pt] at (12,0) {$1$};
     \node[inner sep=2pt] at (0,-2) {$23$};
     \node[inner sep=2pt] at (1.5,-2) {$18$};
     \node[inner sep=2pt] at (3,-2) {$17$};
     \node[inner sep=2pt] at (4.5,-2) {$13$};
     \node[inner sep=2pt] at (6,-2) {$10$};
     \node[inner sep=2pt] at (7.5,-2) {\textcolor{red}{$1$}};
     \node[inner sep=2pt] at (9,-2) {\textcolor{red}{$6$}};
     \node[inner sep=2pt] at (10.5,-2) {\textcolor{red}{$5$}};
     \node[inner sep=2pt] at (12,-2) {\textcolor{red}{$16$}};
     \node[] at (9.75,-2.5) {$\underbrace{\quad\quad\quad\quad\quad\quad\quad\quad\quad\quad\quad\quad\quad}_{\text{the unconnected parts}}$};
     \node[] at (3,-2.5) {$\underbrace{\quad\quad\quad\quad\quad\quad\quad\quad\quad\quad\quad\quad\quad\quad\quad\quad\quad}_{\mu}$};
     \node[] at (6,0.5) {$\overbrace{\quad\quad\quad\quad\quad\quad\quad\quad\quad\quad\quad\quad\quad\quad\quad\quad\quad\quad\quad\quad\quad\quad\quad\quad
     \quad\quad\quad\quad\quad\quad\quad\quad}^{\la}$};
     \end{tikzpicture}
 \end{gather*}
\end{rem}

\begin{cor}\label{t:combin.}
We have the following combinatorial plethystic Murnaghan-Nakayama rule for Schur $Q$-functions:
\begin{align}\label{e:combin.}
(p_s\circ q_k)Q_{\mu}=\sum\limits_{\lambda}sgn(\sigma)2^{A(\lambda/\mu)}2^{l(\mu)-l(\lambda)}Q_{\lambda}
\end{align}
where the sum runs over all strict partitions $\lambda\supset\mu$ such that $\lambda/\mu$ is a symmetric horizontal $(s,k)$-strip, $\sigma$ is the permutation of the reordered composition of $\la$ with respect to $\mu$, and $A(\lambda/\mu)$ is the normalized $a$-number of
the skew diagram $\lambda/\mu$.
\end{cor}
\begin{proof}
Recall that $T^{(s)}_{k}=p_s\circ q_k$. It follows from Corollary \ref{t:cor2}, Theorem \ref{t:iff} and Theorem \ref{t:thm1}.
\end{proof}

When $\mu=\emptyset$, the corollary gives a combinatorial formula for the plethysm $p_s\circ q_k$. Note that
Schur Q-functions have a factorization formula given by Mizukawa \cite{M} which leads to a combinatorial formula for the plethysm $p_{s}\circ Q_{\lambda}$.

\section{\textbf{ A plethystic Murnaghan-Nakayama rule for Hall-Littlewood functions}}
\subsection{\bf {Vertex operator realization of Hall-Littlewood functions}}

Let us recall the vertex operator realization of Hall-Littlewood functions \cite{J2}.
The \textit{vertex operator} $H(z)$ and its adjoint operator $H^*(z)$ are t-parameterized maps $\Lambda_{\mathbb{Q}(t)}\longrightarrow \Lambda_{\mathbb{Q}(t)}[[z, z^{-1}]]=\Lambda_{\mathbb{Q}(t)}\otimes \mathbb{Q}(t)[[z,z^{-1}]]$ given by
\begin{align}
\label{e:hallop}
H(z)&=\mbox{exp} \left( \sum\limits_{m\geq 1} \dfrac{1-t^{m}}{m}p_mz^{m} \right) \mbox{exp} \left( -\sum \limits_{m\geq 1} \frac{\partial}{\partial p_m}z^{-m} \right) \notag
=\sum_{m\in\mathbb Z}H_mz^{m},\\
H^*(z)&=\mbox{exp} \left(-\sum\limits_{m\geq 1} \dfrac{1-t^{m}}{m}p_mz^{m} \right) \mbox{exp} \left(\sum \limits_{m\geq 1} \frac{\partial}{\partial p_m}z^{-m} \right) \notag
=\sum_{m\in\mathbb Z}H^*_mz^{-m}.
\end{align}
In particular, when the operator $H(z)$ acts on the vacuum vector $1$,
\begin{align}
H(z).1=\mbox{exp}\left(\sum_{m\geq1}\frac{(1-t^{m})p_{m}}{m}z^{m}\right)=\sum_{m\geq0}q_{m}(t)z^{m}
\end{align}
where $q_{m}(t)$ is the Hall-Littlewood polynomial (in the $p_n$) associated with one-row partition $(m)$:
\begin{align}\label{e:q_{m}}
H_{m}.1=q_{m}(t)=\sum_{\lambda\vdash m}\frac{p_{\lambda}}{z_{\lambda}(t)}.
\end{align}
Similarly, for $m\geq0$ we have that
\begin{align}\label{H^{*}}
H^{\ast}_{-m}.1=\sum_{\lambda\vdash m}\frac{(-1)^{l(\lambda)}}{z_{\lambda}(t)}p_{\lambda}.
\end{align}

More generally, for a partition $\lambda=(\lambda_{1},\ldots ,\lambda_{l})$,
the vertex operator product  $H_{\lambda_{1}}\cdots H_{\lambda_{l}}. 1$ is the
Hall-Littlewood function $Q_{\la}(t)$ \cite[Prop. 2.17]{J2}:
\begin{equation}\label{e:HL}
H_{\lambda_{1}}\cdots H_{\lambda_{l}}. 1=Q_{\la}(t)=
\prod\limits_{i<j} \dfrac{1-R_{ij}}{1-tR_{ij}}q_{\lambda_{1}}\cdots q_{\lambda_{l}}
\end{equation}
where the raising operator $R_{ij}q_{\la}=q_{(\la_{1},\ldots ,\la_{i}+1,\ldots ,\la_{j}-1,\ldots , \la_{l})}$.

Moreover, $H_{\lambda}.1=H_{\lambda_1}\cdots H_{\lambda_l}.1$ are orthogonal in $\Lambda_{\mathbb{Q}(t)}$ \cite[Prop. 3.9]{J2}
\begin{align}\label{e:orth}
\langle H_{\lambda}.1, H_{\mu}.1\rangle_t=\delta_{\lambda\mu}b_{\lambda}(t),
\end{align}
where $b_{\lambda}(t)=(1-t)^{l(\lambda)}\prod_{i\geqslant 1}[m_i(\lambda)]!$ and $[n]=\frac{1-t^n}{1-t}$, $[n]!=[n][n-1]\cdots[1]$.

We collect the relations of the vertex operators as follows.
\begin{lem} \cite[Prop. 2.12, Lem. 2.16]{J2} The operators $H_n$ and $H_n^*$ satisfy the following relations
\begin{align}\label{e:com1}
H_{m}H_n-tH_nH_m&=tH_{m+1}H_{n-1}-H_{n-1}H_{m+1}\\ \label{e:com2}
H^*_{m}H^*_n-tH^*_nH^*_m&=tH^*_{m-1}H^*_{n+1}-H^*_{n+1}H^*_{m-1}\\ \label{e:com3}
H_{m}H^*_n-tH^*_nH_m&=tH_{m-1}H^*_{n-1}-H^*_{n-1}H_{m-1}+(1-t)^2\delta_{m, n}\\
H_{-n}. 1&=\delta_{n, 0}, \quad H_{n}^{*}. 1=\delta_{n, 0},\quad n\geq0.
\end{align}
where $\delta_{m, n}$ is the Kronecker delta function.
\end{lem}

In general, expressing $H_{\mu}$ for a composition $\mu$ in terms of the basis elements $H_{\lambda}$, $\lambda\in\mathcal P$
can be formulated as follows.
Define the transformation $S_{i,a}$ acting on composition $\mu=(\mu_1, \cdots, \mu_i, \mu_{i+1}, \cdots)\\(\mu_{i+1}>\mu_i)$ by $$S_{i,a}(\mu_1, \cdots, \mu_i, \mu_{i+1}, \cdots)=(\mu_1, \cdots, \mu_{i+1}-a, \mu_{i}+a, \cdots),$$
where $a\leq \lfloor (\mu_{i+1}-\mu_i)/2\rfloor$. And the action of $S_{i,a}$ on $H_{\mu}$ is defined by $S_{i,a}H_{\mu}=H_{S_{i,a}\mu}$.

Given a composition $\mu$, denote the coefficient of $S_{i,a}H_{\mu}$ in $H_{\mu}$ by
\begin{equation}\label{e:straight}
C(S_{i,a})=\begin{cases} t & a=0\\ t^{a+1}-t^{a-1} & 1\leq a< \lfloor\frac{\mu_{i+1}-\mu_i}2\rfloor \\ t^{a+\epsilon}-t^{a-1} &  a= \lfloor\frac{\mu_{i+1}-\mu_i}2\rfloor
\end{cases}
\end{equation}
where $0\leq\epsilon\leq1$ with $\epsilon\equiv \mu_{i+1}-\mu_{i} (mod\, 2)$. If $\mu_i<\mu_{i+1}$, then it follows from \cite[p.214, Ex.2]{Mac} that
\begin{equation}\label{e:transf}
H_{\mu}=\sum_{a=0}^{\lfloor\frac{\mu_{i+1}-\mu_i}2\rfloor}C(S_{i, a})H_{S_{i, a}\mu}
\end{equation}
Observe that the $i$th and $(i+1)$th parts on the right-hand side are in the right order. Thus we can repeatedly use \eqref{e:transf} to regularize
composition $\mu$ into a partition as follows.

We call a sequence of compositions $(\mu^{(0)}, \mu^{(1)}, \ldots, \mu^{(r)})$ a {\it straightening path} from a composition $\mu=\mu^{(0)}$ to a partition $\mu^{(r)}\in\mathcal P$
if all compositions $\mu^{(k)}$ are proper compositions except the last one and $\mu^{(k)}=S_{i_k, a_k}\mu^{(k-1)}$ for some nonnegative $a_k\leq \lfloor (\mu^{(k-1)}_{i_k+1}-\mu^{(k-1)}_{i_k})/2\rfloor$, $k=1, \ldots, r$.
Here a proper composition means that there exists 
some inversion among its parts, i.e. it is not a partition.

For a straightening path $\underline{\mu}=(\mu^{(0)}, \mu^{(1)}, \ldots, \mu^{(r)})$ from a composition $\mu=\mu^{(0)}$ to a partition $\la=\mu^{(r)}$, we define
\begin{equation}\label{e:proc}
C(\underline{S_{\mu}})=C(S_{i_r, a_r})\cdots C(S_{i_2, a_2})C(S_{i_1, a_1})
\end{equation}
where $S_{i_k, a_k}\mu^{(k-1)}=\mu^{(k)}$, $k=1, \ldots, r$.



In particular, when $t=0$, $C(S_{\underline{i}, \underline{a}})=0$  unless all $a_i=1$, and $\mu_{i+1}-\mu_i\geq 2$, in
which case $C(S_{\underline{i}, \underline{1}})=(-1)^r$.
When $t=-1$, $C(S_{\underline{i}, \underline{a}})=0$ unless all $a_i=0$, in which case $C(S_{\underline{i}, \underline{0}})=(-1)^r$.
(Since Schur $Q$-functions $Q_{\lambda}=Q_{\lambda}(t)|_{t=-1}$, combining with \cite[Prop. 4.15]{J1}, we have
$H_{\frac{\mu_{i+1}+\mu_{i}}{2}}H_{\frac{\mu_{i+1}+\mu_{i}}{2}}\mid_{t=-1}=Q_{\frac{\mu_{i+1}+\mu_{i}}{2}}Q_{\frac{\mu_{i+1}+\mu_{i}}{2}}=0$ when $\mu_{i+1}-\mu_{i}$ is even
and $a=\frac{\mu_{i+1}-\mu_{i}}{2}$).


Let $\mu$ be a composition and $\lambda$ be a partition. Denote
\begin{align}\label{e:def}
B(\lambda, \mu)\triangleq \sum\limits_{\underline{\mu}}C(S_{\underline{\mu}})
\end{align}
summed over those straightening paths $\underline{\mu}=(\mu^{(0)}, \mu^{(1)}, \ldots, \mu^{(r-1)}, \mu^{(r)})$ from the composition $\mu=\mu^{(0)}$ to the partition $\la=\mu^{(r)}$
with
each step $\mu^{(k)}=S_{i_k, a_k}(\mu^{(k-1)})$ being chosen at the minimum index $i_k$ among all adjacent inversions of 
$\mu^{(k-1)}$.  Such a straightening path is called a canonical one.

For example, suppose $\mu=(8,7,2,5,6)$. 
The first adjacent inversion of $\mu$ occurs at the $3$rd index,
so $\mu$ is changed to two compositions $(8, 7, 5, 2, 6)$ and $(8, 7, 4, 3, 6)$ by $S_{3, 0}$ and $S_{3, 1}$ respectively.
For the composition $(8, 7, 5, 2, 6)$ the first adjacent inversion appears at the $4$th index, so it is transformed
to compositions $(8, 7, 5, 6, 2), (8, 7, 5, 5, 3)$ and
$(8, 7, 5, 4, 4)$ by $S_{4, 0}, S_{4, 1}, S_{4, 2}$ respectively. One can continue this process until all resulting compositions are partitions,
then $H_{\mu}$ is expressed as a linear combinations of basis elements $H_{\la}$, $\la\in\mathcal P$.
See the following tree diagram for this process which shows all canonical straightening paths of $\mu$.
Here we write $(a,b,c,d,e)$ for $H_{(a,b,c,d,e)}$ and the red $(i,j)$ on arrows stand for $S_{i,j}$.
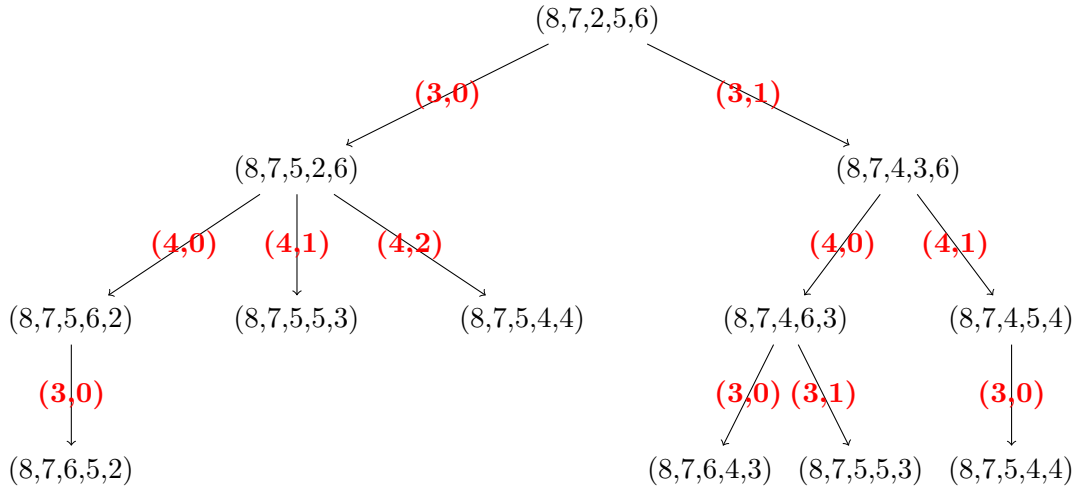
\begin{gather*}
  \centering
\begin{tikzpicture}
\node[] (1) at(5,0) {(8,7,2,5,6)};
\node[] (2) at(1,-2) {(8,7,5,2,6)};
\node[] (3) at(9,-2) {(8,7,4,3,6)};
\node[] (4) at(-2,-4) {(8,7,5,6,2)};
    \node[] (5) at(1,-4) {(8,7,5,5,3)};
    \node[] (6) at(4,-4) {(8,7,5,4,4)};
    \node[] (7) at(7.5,-4) {(8,7,4,6,3)};
    \node[] (8) at(10.5,-4) {(8,7,4,5,4)};
    \node[] (9) at(-2,-6) {(8,7,6,5,2)};
    \node[] (10) at(6.5,-6) {(8,7,6,4,3)};
    \node[] (11) at(8.5,-6) {(8,7,5,5,3)};
    \node[] (12) at(10.5,-6) {(8,7,5,4,4)};
\draw[->] (1)--(2);
\draw[->] (1)--(3);
\draw[->] (2)--(4);
\draw[->] (2)--(5);
\draw[->] (2)--(6);
\draw[->] (3)--(7);
\draw[->] (3)--(8);
\draw[->] (4)--(9);
\draw[->] (7)--(10);
\draw[->] (7)--(11);
\draw[->] (8)--(12);
\node at(3,-1) {\textcolor{red}{\bf(3,0)}};
\node at(7,-1) {\textcolor{red}{\bf(3,1)}};
\node at(-0.5,-3) {\textcolor{red}{\bf(4,0)}};
\node at(1,-3) {\textcolor{red}{\bf(4,1)}};
\node at(2.5,-3) {\textcolor{red}{\bf(4,2)}};
\node at(8.25,-3) {\textcolor{red}{\bf(4,0)}};
\node at(9.75,-3) {\textcolor{red}{\bf(4,1)}};
\node at(-2,-5) {\textcolor{red}{\bf(3,0)}};
\node at(7,-5) {\textcolor{red}{\bf(3,0)}};
\node at(8,-5) {\textcolor{red}{\bf(3,1)}};
\node at(10.5,-5) {\textcolor{red}{\bf(3,0)}};
\end{tikzpicture}
\end{gather*}

In particular, there are two canonical straightening paths from composition $\mu=(8,7,2,5,6)$ to partition $\la=(8,7,5,4,4)$: $(\mu,(8,7,5,2,6),\la)$ and $(\mu,(8,7,4,3,6),(8,7,4,5,4),\la)$. 
So,
$$B(\la,\mu)=C(S_{4,2})C(S_{3,0})+C(S_{3,0})C(S_{4,1})C(S_{3,1})=t(t-1)\cdot t+t\cdot (t^2-1)\cdot (t^2-1)=t^5-t^3-t^2+t.$$


In summary, 
the following relation can be derived by using \eqref{e:transf} and \eqref{e:def}.

\begin{prop}  \label{p:straight}\cite[Prop. 3.1]{JL2} Suppose $\mu$ is a composition, then
\begin{equation}\label{e:straight2}
H_{\mu}=\sum\limits_{\lambda} B(\lambda,\mu) H_{\la}.
\end{equation}
summed over the partitions $\lambda$ obtainable by canonical straightening paths from the composition $\mu$.
\end{prop}

We remark that the condition of fixing each $i_j$ ($j=1, \ldots, r$) in \eqref{e:proc} as the minimal adjacent inversion index along the straightening paths was missing in the statement in \cite[Prop. 3.1]{JL2}.
We also note that $\lambda$ appears only when $\lambda\geq\mu$ in \eqref{e:straight2}
in the dominance order. This is due to the fact that the transformation $S_{i, a}\mu \geq \mu$ for any composition
$\mu$ (Here we extend the dominance order to compositions).

\subsection{\textbf{ A plethystic Murnaghan-Nakayama rule for Hall-Littlewood functions}}

For a fixed positive integer $s$, we introduce the \textit{vertex operator} $L^{(s)}(z)$ and its adjoint operator $L^{(s)\ast}(z)$ as follows:
\begin{align}\label{e:Lop}
L^{(s)}(z)&=\mbox{exp}\left(\sum_{m\geq1}\frac{(1-t^{sm})p_{sm}}{m}z^{m}\right)=\sum_{m\geq0}L^{(s)}_{m}z^{m},\\
L^{(s)\ast}(z)&=\mbox{exp}\left(\sum_{m\geq1}s\frac{\partial}{\partial p_{sm}}z^{-m}\right)=\sum_{m\geq0}L^{(s)\ast}_{m}z^{-m}.
\end{align}
It follows from \eqref{e:Lop} that
\begin{align}\label{e:equal}
L^{(s)}_{m}=\sum_{\lambda\vdash m}\frac{p_{s\lambda}}{z_{\lambda}(t^{s})}=\sum_{\lambda\vdash m}\frac{\prod_{i=1}^{l(\lambda)}(1-t^{s\lambda_{i}})p_{\lambda}(x^{s}_{1},x^{s}_{2},\cdots)}{z_{\lambda}}=q_{m}(x^{s}_{1},x^{s}_{2},\cdots;t^{s})=p_{s}\diamond q_{m}(t)
\end{align}
where $s\lambda=(s\lambda_1, s\lambda_2,\cdots).$
\begin{exmp} As an example, we compute $p_{2}\diamond q_{m}(t)$ in terms of $q_i(t), i=0,1,\ldots,2m.$ Indeed, by definition of $L^{(s)}(z)$, we have
\begin{align*}
L^{(2)}(z^2)=&\mbox{exp}\left(\sum_{m\geq1}\frac{(1-t^{2m})}{m}p_{2m}z^{2m}\right)\\
=&\mbox{exp}\left(\sum_{m\geq1}\frac{(1-t^{m})(1+(-1)^{m})}{m}p_{m}z^{m}\right)\\
=&L^{(1)}(z)L^{(1)}(-z).
\end{align*}
Therefore, we have
\begin{align}\label{e:p2qm}
p_{2}\diamond q_{m}(t)=L^{(2)}_{m}=\sum_{i=0}^{2m}(-1)^iq_{i}(t)q_{2m-i}(t).
\end{align}
\end{exmp}

\begin{prop}\label{t:HL} For integer $k$ and $m$,
\begin{align}\label{e:H^{*}L}
L^{(s)\ast}_{k}H_{m}=H_{m}L^{(s)\ast}_{k}+(1-t^{s})\sum_{r=1}^{k}H_{m-sr}L^{(s)\ast}_{k-r}.
\end{align}
\begin{proof} Similar to the proof of Theorem \ref{t:Q}, using usual techniques of vertex operators we have
\begin{align*}
L^{(s)\ast}(z)H(w)&=H(w)L^{(s)\ast}(z)\exp\left(\sum_{m\geq1}\frac{1-t^{sm}}{m}(\frac{w^{s}}{z})^{m}\right)\\
&=H(w)L^{(s)\ast}(z)\frac{z-(tw)^{s}}{z-w^{s}}\\
&=H(w)L^{(s)\ast}(z)(1+(1-t^{s})\sum_{n=1}^{\infty}(\frac{w^{s}}{z})^{n})
\end{align*}
where the rational functions are expanded at $w=0$. The relation then follows by comparing coefficients of $z^{-k}w^{m}$.
\end{proof}
\end{prop}
Using Proposition \ref{t:HL}, we immediately get the following result.
\begin{thm}\label{t:H-L} For a partition $\lambda=(\lambda_{1},\lambda_{2},\ldots ,\lambda_{l})\vdash n$, and integers $s, k$
\begin{align} \label{e:H-L}
L_{k}^{(s)\ast}H_{\lambda}.1=\sum_{\nu\vDash k}(1-t^{s})^{l(\nu)}H_{\lambda-s\nu}.1.
\end{align}
\end{thm}
\begin{proof} It is shown by induction on $l(\lambda)$ using \eqref{e:H^{*}L}.
\end{proof}

By Proposition \ref{p:straight}, we can rewrite \eqref{e:H-L} as
\begin{align}\label{e:eq5}
L_{k}^{(s)\ast}H_{\lambda}.1=\sum_{\mu\vdash n-sk}\sum_{\nu\vDash k}(1-t^{s})^{l(\nu)}B(\mu,\lambda-s\nu)H_{\mu}.1.
\end{align}

Now we give the plethystic Murnaghan-Nakayama rule for Hall-Littlewood functions.
\begin{thm} \label{t:p H-L} Let $s, k$ be two positive integers and $\mu$ be a partition of $n$, then we have
\begin{align}
(p_{s}\diamond q_{k}(t))H_{\mu}.1=\sum_{\mu\subset\lambda\vdash n+sk}\sum_{\nu\models k}\frac{b_{\mu}(t)}{b_{\lambda}(t)}(1-t^s)^{l(\nu)}B(\mu,\lambda-s\nu)H_{\lambda}.1.
\end{align}
\end{thm}
\begin{proof}
Suppose $c_{\lambda}(t)$ is the coefficient of $H_{\lambda}.1$ in $(p_{s}\diamond q_{k}(t))H_{\mu}.1$, then
\begin{align*}
c_{\lambda}(t)=&\frac{1}{b_{\lambda}(t)}\langle (p_{s}\diamond q_{k}(t))H_{\mu}.1, H_{\lambda}.1 \rangle\quad\quad (\text {by \eqref{e:orth}})\\
=&\frac{1}{b_{\lambda}(t)}\langle H_{\mu}.1, L^{(s)*}_{k}H_{\lambda}.1 \rangle\quad\quad\quad (\text { by  \eqref{e:equal})} \\
=&\frac{b_{\mu}(t)}{b_{\lambda}(t)}\sum_{\nu\models k}(1-t^s)^{l(\nu)}B(\mu,\lambda-s\nu)\quad (\text {by \eqref{e:orth} and \eqref{e:eq5}}).
\end{align*}
\end{proof}

\begin{cor} \label{t:cor1}For any positive integers $s, k$, we have that
\begin{align}
p_{s}\diamond q_{k}(t)=\sum_{\lambda\vdash sk}\sum_{\nu\models k}\frac{1}{b_{\lambda}(t)}(1-t^s)^{l(\nu)}B(\emptyset,\lambda-s\nu)H_{\lambda}.1.
\end{align}
\end{cor}

\begin{exmp}\label{t:p2qk(t)} Set $s=2$, $\mu=\emptyset$. Due to \eqref{e:p2qm}, $H_{\lambda}.1$ with $l(\lambda)\geq 3$ do not appear in the expansion of $p_2\diamond q_{k}(t).$ So it is sufficient to compute $B(\emptyset, \lambda-2\nu)$ when $l(\lambda)\leq 2.$ In this case, $\lambda=(2k-i, i) (i=0,1,\ldots,k).$
Using case by case consideration, we have
\begin{align*}
c_{\lambda}(t)=\sum_{\nu\models k}\frac{1}{b_{(2k-i,i)}(t)}(1-t^2)^{l(\nu)}B(\emptyset,(2k-i,i)-2\nu)=
\begin{cases}
(-1)^k& \text {if $i=k$}\\
(-1)^i(t+1)& \text {if $0\leq i<k.$}
\end{cases}
\end{align*}
Therefore,
\begin{align}\label{e:p2qk}
p_{2}\diamond q_{k}(t)=(-1)^{k}H_{(k,k)}.1+(t+1)\sum_{i=0}^{k-1}(-1)^{i}H_{(2k-i,i)}.1.
\end{align}
\end{exmp}

This is a generalization of Littlewood's result \cite[p.122]{L2} for Schur functions:
\begin{align*}
p_2\circ h_k=\sum_{j=0}^{k}(-1)^js_{(2k-j,j)}
\end{align*}
where $h_k$ (resp. $s_{(2k-j,j)}$) is the complete (resp. Schur) symmetric functions.

\section{Appendix}
{\em Proof of \eqref{e:decomposition2}:} Suppose $B=(b_{ij})_{2n\times2n}$ is an antisymmetric matrix obtained from $A=(a_{ij})_{2n\times2n}$ by rotating the first $i$ rows and columns by the cyclic permutation $\left(\begin{array}{ccccc}
1&2&\cdots&i-1&i\\
2&3&\cdots&i&1\\
\end{array}\right).$ Explicitly
$$
b_{1j+1}=
\begin{cases}
a_{ij}& \text{if $j=1,2,\cdots,i-1$;}\\
a_{i,j+1}& \text{if $j=i,i+1,\cdots,2n$.}
\end{cases}
\quad
B_{1j+1}=
\begin{cases}
A_{ij}& \text{if $j=1,2,\cdots,i-1$;}\\
A_{i,j+1}& \text{if $j=i,i+1,\cdots,2n$.}
\end{cases}
$$
Therefore $\text{Pf(A)}=(-1)^{i-1}\text{Pf(B)}.$ It follows from \eqref{e:expansion} that
\begin{align*}
\text{Pf(B)}&=\sum_{j=2}^{2n}(-1)^{j}b_{1j}\text{Pf($B_{1j}$)}\\
&=\sum_{j=1}^{i-1}(-1)^{j+1}b_{1j+1}\text{Pf($B_{1j+1}$)}+\sum_{j=i+1}^{2n}(-1)^{j}b_{1j}\text{Pf($B_{1j}$)}\\
&=\sum_{j=1}^{i-1}(-1)^{j+1}a_{ij}\text{Pf($A_{ij}$)}+\sum_{j=i+1}^{2n}(-1)^{j}a_{ij}\text{Pf($A_{ij}$)}\\
&=\sum_{j=1}^{i-1}(-1)^{j}a_{ji}\text{Pf($A_{ij}$)}+\sum_{j=i+1}^{2n}(-1)^{j}a_{ij}\text{Pf($A_{ij}$)}\\
&=\sum\limits_{j\neq i}(-1)^{j}:a_{ij}:\text{Pf($A_{ij}$)}.
\end{align*}
\hfill $\Box$

\bigskip

\bigskip
\noindent{\bf Acknowledgments}

We would like to thank the anonymous referee for helpful comments and correcting an error in the original statement of Prop. \ref{p:straight}.
The project is partially supported by Simons Foundation under
grant no. 523868, NSFC grant 12171303, and the Humboldt foundation.
The second author also thanks Max-Planck Institute for Mathematics in the Sciences, Leipzig
for hospitality during the project. The third author is supported by China Scholarship Council and he also acknowledges the hospitality of the Faculty of Mathematics of University of Vienna, where parts of the present work was completed.

\bigskip

\noindent{\bf Statement of Conflict Interest.} On behalf of all authors, the corresponding author states that there is no conflict of interest.

\bigskip

\noindent{\bf Data Availability Statement.} All data generated during the study are included in the article.

\bigskip

\bibliographystyle{plain}

\begin{thebibliography}{6}

\bibitem{Bak} T. H. Baker, {\em Symmetric function products and plethysms and the boson-fermion correspondence}. J. Phys. A 28 (1995), 
589-606.

\bibitem{CJL} Y. Cao, N. Jing, N. Liu, {\em Plethystic Murnaghan-Nakayama rule via vertex operators}, arXiv:2212.08412.

\bibitem{DJKM} E.~Date, M.~Jimbo, M.~Kashiwara, T.~Miwa, {\em Transformation groups for soliton equations. IV. A new type of soliton equations of KP-type}, Phys. D 4 (1981/82), 
    343-365.

\bibitem{DLT} J. D\'esarm\'enien, B. Leclerc and J. Y. Thibon, {\em Hall-Littlewood functions and Kostka-Foulkes polynomials in representation  theory}. S\'em. Lothar. Combin. 32 (1994), Art. B32c, 38 pp.

\bibitem{EPW} A. Evseev, R. Paget and M. Wildon, {\em Character deflations and a generalization of the Murnaghan-Nakayama rule.} J. Group Theory 17 (2014), no. 6, 1035-1070.

\bibitem{H} B. Hall, {\em Lie groups, Lie algebras, and representations.
An elementary introduction}. Second edition. Graduate Texts in Mathematics, 222. Springer, Cham, 2015. xiv+449 pp.

\bibitem{J1} N. Jing, {\em Vertex operators, symmetric functions and the spin group $\Gamma _n$}, J. Algebra 138 (1991), 340-398.

\bibitem{J2} N. Jing, {\em Vertex operators and Hall-Littlewood symmetric functions}, Adv. Math. 87 (1991), 226-248.

\bibitem{JL2} N. Jing, N. Liu, {\em The Green polynomials via vertex operator}, J. Pure Appl. Algebra 226 (2022), 
Paper No. 107032.

\bibitem{L1} D. E. Littlewood, {\em The theory of group characters and matrix representations of groups}. Oxford University Press, New York 1940. viii+292 pp.

\bibitem{L2} D. E. Littlewood, {\em On orthogonal and symplectic group characters}. J. London Math. Soc. 30, (1955). 121-122.

\bibitem {Mac} I. G. Macdonald, {\em Symmetric functions and Hall polynomials}, 2nd ed., Clarendon Press, Oxford, 1995.

\bibitem {M} H. Mizukawa, {\em Factorization of Schur's Q-functions and plethysm}, Ann. Comb. 6 (2002), 87-101.

\bibitem{S} I. Schur, {\em \"Uber die Darstellung der symmetrischen und der alternierenden Gruppe
durch gebrochene linear Substitutionen}, J. Reine Angen. Math. 139 (1911),
155-250.

\bibitem {St} J. R. Stembridge, {\em Shifted tableaux and the projective representations of symmetric
groups}, Adv. Math. 74 (1989), 87-134.

\bibitem{W} M. Wildon, {\em A combinatorial proof of a plethystic Murnaghan-Nakayama rule,} SIAM J. Discrete Math. 30 (2016), 
1526-1533.
\end{thebibliography}

\end{document}